\documentclass[10pt]{amsart}

\usepackage{amsmath}
\usepackage{amsthm}
\usepackage{amsopn}
\usepackage{amssymb}
\usepackage[all]{xy}
\usepackage{lscape,xcolor}
\usepackage{graphicx}
\usepackage{mathtools}
\usepackage{hang}
\usepackage{mathrsfs}
\usepackage{tikz-cd}
\usetikzlibrary{arrows}
\usepackage{nicefrac}
\usepackage{multicol}
\usepackage{graphicx}
\usepackage{hyperref}
\usepackage{bm}

\allowdisplaybreaks[1]

\newtheorem{theorem}{Theorem}[section]
\newtheorem{cor}[equation]{Corollary}
\newtheorem{lemma}[equation]{Lemma}
\newtheorem{proposition}[equation]{Proposition}

\newtheorem{construction}[equation]{Construction}
\newtheorem{remark}[equation]{Remark}
\newtheorem{definition}[equation]{Definition}
\newtheorem{example}[equation]{Example}

\newtheorem{thm}{Theorem}
\newtheorem*{thm*}{Theorem}

\newtheorem*{cor*}{Corollary}
\newtheorem*{lem*}{Lemma}
\newtheorem*{prop*}{Proposition}

\theoremstyle{definition}

\newtheorem*{defn*}{Definition}
\newtheorem*{ex*}{Example}
\newtheorem*{exs*}{Examples}
\newtheorem*{rmk*}{Remark}
\newtheorem*{claim*}{Claim}

\numberwithin{equation}{section}
\numberwithin{figure}{section}

\newcommand{\mC}{{\mathbb C}}

\newcommand{\mF}{{\mathbb F}}

\newcommand{\mR}{{\mathbb R}}

\newcommand{\mZ}{{\mathbb Z}}

\newcommand{\cN}{{\mathcal N}}

\newcommand{\umF}{\underline{\mF}}
\newcommand{\umZ}{\underline{\mZ}}

\newcommand{\vertsimeq}{\rotatebox{90}{$\,\simeq$}}

\definecolor{green}{HTML}{009E73}
\definecolor{blue}{HTML}{0072B2}
\definecolor{yellow}{HTML}{E69F00} 
\definecolor{rose}{HTML}{DE087F}

\DeclarePairedDelimiter\ceil{\lceil}{\rceil}
\DeclarePairedDelimiter\floor{\lfloor}{\rfloor}

\title{The $H \underline{\mathbb{F}}_2$-homology of $C_2$-equivariant Eilenberg-MacLane spaces}
\author{Sarah Petersen}
\address{
	Dept. of Mathematics \\
	University of Notre Dame \\
	Notre Dame, IN, U.S.A.
}

\begin{document}
	
	\begin{abstract}
		We extend Ravenel-Wilson Hopf ring techniques to $C_2$-equivariant homotopy theory. Our main application and motivation is a computation of the $RO(C_2)$-graded homology of $C_2$-equivariant Eilenberg-MacLane spaces. The result we obtain for $C_2$-equivariant Eilenberg-MacLane spaces associated to the constant Mackey functor $\umF_2$ gives a $C_2$-equivariant analogue of the classical computation due to Serre. We also investigate a twisted bar spectral sequence computing the homology of these equivariant Eilenberg-MacLane spaces and suggest the existence of another twisted bar spectral sequence with $E^2$-page given in terms of a twisted Tor functor.
	\end{abstract}

\maketitle
\tableofcontents

\section{Introduction}\label{sec:intro}

Computations of invariants in equivariant homotopy theory have powerful applications contributing to solutions of outstanding classification problems in geometry, topology, and algebra. A primary example is Hill, Hopkins, and Ravenel's solution of the Kervaire invariant one problem \cite{HillHopkinsRavenel2016}, which used computations in equivariant homotopy theory to answer the question of when a framed $4k + 2$ dimensional manifold can be surgically converted into a sphere. Despite the success of numerous applications, many equivariant computations remain difficult to access due to their rich structure. This is especially true for (unstable) equivariant spaces, for which many computations have not yet been completed, despite their analogous nonequivariant results being well known. 

This paper extends Ravenel-Wilson Hopf ring techniques \cite{RavenelWilson1977, RavanelWilson1980, Wilson1982} to $C_2$-equivariant homotopy theory. Our main application and motivation is a computation of the $RO(C_2)$-graded homology of $C_2$-equivariant Eilenberg-MacLane spaces. The result, stated over the course of Theorems \ref{Thm:signed}, \ref{Thm:HKVgen}, and \ref{Thm:sigmaPlusl}, is a $C_2$-equivariant analogue of the classical cohomology computation completed by Serre in the 1950s \cite{serre1953}. 

Non-equivariantly, Serre applied the Borel theorem \cite[page 88, Theorem 1]{moshertangora1968} to the path space fibration 
$$K (\mF_p, n) \simeq \Omega K(\mF_p, n + 1) \to  P \big( K(\mF_p, n + 1) \big) \to K(\mF_p, n + 1),$$ to calculate the cohomology of $K(\mF_p, n + 1)$ given $H^* K(\mF_p, n).$ 
In $C_2$-equivariant homotopy theory, the constant Mackey functor $\umF_2$ is the analogue of the group $\mF_2$ and the Eilenberg-MacLane spaces $K_V = K( \umF_2, V)$ are graded on the real representations $V$ of the group $C_2$ rather than on the integers. Since the group $C_2$ has two irreducible real representations, the trivial representation and the sign representation $\sigma,$ the analogous equivariant computation would require computing the cohomology of $K_{V + \sigma}$ from $H^\star K_{V}$ in addition to $H^\star K_{V + 1}$ from $H^\star K_V.$ This would necessitate having a so called signed or twisted version of the Borel theorem. However, no such theorem is known to exist, making it difficult to study the cohomology of the spaces $K_{V + \sigma}$ with these techniques. We call $K_{V + \sigma}$ a signed delooping of $K_V$ since the space of signed loops $\Omega^\sigma K_{V + \sigma} \simeq K_V.$ 

While direct extension of Serre's original argument does not allow for the computation of the cohomology of signed deloopings, it has been successfully applied to study trivial representation deloopings of $K_\sigma,$ whose cohomology is known \cite{HuKriz2001}. This approach is described in Ugur Yigit's thesis \cite{yigit2019}, where it is noted that the $RO(C_2)$-graded cohomology of all $C_2$-equivariant Eilenberg-MacLane spaces $K_{\sigma + *}$ can be computed using this method. Throughout, we use $*$ to denote integer grading and reserve $\star$ to denote grading by finite dimensional real representations. 

A major reason to study Ravenel-Wilson Hopf ring techniques in $C_2$-equivariant homotopy theory is that they provide a way to study $\sigma$-deloopings. These techniques, which investigate multiplicative structures coming from $H$-space maps on spaces having a graded multiplication, lend additional structure that can be exploited to complete computations. 

An important tool in classical applications of Ravenel-Wilson Hopf ring techniques is the bar construction $B.$ This construction plays a significant role in computation because $B$ is a trivial representation delooping functor with $B K_V \simeq K_{V + 1}$. In the $C_2$-equivariant world, there is a twisted bar construction $B^\sigma,$ which is a sign representation delooping functor with $B^\sigma K_V \simeq K_{V + \sigma}$ \cite{Liu2020}. We use these two constructions to explicitly model multiplicative structures on the spaces $K_V$ at the point set level (Theorem \ref{thm:5.1}), directly extending work by Ravenel and Wilson \cite{RavanelWilson1980}. We also describe our approach to using this structure to investigate signed and trivial representation deloopings in Section \ref{sec:multStrs}.

Whereas Ravenel and Wilson use a collapsing integer-graded bar spectral sequence to inductively compute the homology of classical nonequivariant Eilenberg-MacLane spaces \cite{Wilson1982}, we deduce many of our equivariant computations from non-equivariant ones using a computational method introduced by Behrens-Wilson (\cite{BehrensWilson2018} Lemma 2.8). Starting with the $RO(C_2)$-graded homology of $K_\sigma$, we use the graded multiplication on the spaces $K_V$ coming from the genuine equivariant ring structure on $H \umF_2,$ to produce elements of the $RO(C_2)$-graded homology of $K_{* \sigma}.$ We then use the point set level understanding of multiplicative structures on the spaces $K_{*\sigma}$ developed in Theorem \ref{thm:5.1} to verify that these elements in fact form a free basis for the homology.

Once we have computed $H_\star K_{* \sigma}$ (Theorem \ref{Thm:signed}), we use Hopf ring structures in $RO(C_2)$-graded bar spectral sequences to inductively compute $H_\star K_{i \sigma + j}$ from $H_\star K_{i \sigma}$ (Theorem \ref{Thm:HKVgen}). In the case where $i = 1$, that is for the spaces $K_{\sigma +*},$ we name all homology generators in terms of the Hopf ring structure (Theorem \ref{Thm:sigmaPlusl}). The task of naming homology generators for the spaces $K_V,$ where $\sigma +1 \subset V,$increases in complexity as the number of sign representations increases. We illustrative this phenomenon in Section \ref{sec:btwbssComps}.

Knowing the $RO(C_2)$-graded homology of the $C_2$-equivariant Eilenberg-MacLane spaces $K_V,$ we turn to investigating the $RO(C_2)$-graded twisted bar spectral sequence. Much like the classical integer graded bar spectral sequence, the $RO(C_2)$-graded twisted bar spectral sequence arises from a filtered complex. However, computations with this twisted spectral sequence are more complicated than in the classical case. For example, in contrast to the classical case where the integer-graded bar spectral sequence computing the non-equivariant mod $p$ homology of the classical Eilenberg-MacLane spaces $K_* = K(\mF_p, *)$ collapses on the $E^2$-page \cite{Wilson1982}, we find there are arbitrarily long equivariant degree shifting differentials, similar to those observed in Kronholm's study of the cellular spectral sequence \cite{kronholm2010}, in the $RO(C_2)$-graded twisted bar spectral sequences computing the homology of the signed representation spaces $K_{n\sigma},$ where $n \geq 2$.

While the $RO(C_2)$-graded twisted bar spectral sequence is quite complicated in general, the differentials and extensions appear to arise in an extremely structured way, governed by a norm structure. We use our knowledge of $H_\star K_{* \sigma}$ and the $E^\infty$-page to deduce information about the $RO(C_2)$-graded twisted bar spectral sequences computing the homology of $K_{* \sigma}.$ This allows us to write down conjectures concerning many of the differentials in Section \ref{sec:btwbssComps}. Our equivariant computations show that, unlike in the non-equivariant integer graded situation, the $RO(C_2)$-graded twisted bar spectral sequences computing $H_\star K_{n \sigma},$ where $n \geq 2,$ have a rich structure quite distinct from the collapsing bar spectral sequence in the classical nonequivariant case \cite{Wilson1982}. Differences between integer graded and $RO(C_2)$-graded bar and twisted bar spectral sequences are discussed in Section \ref{sec:btwbssComps}.

In parallel with calculating the homology of a space, the corresponding computational tools are worth investigating in a purely algebraic setting. This study of the homological algebra involved produces tools which can also be applied in settings outside of topology. One example of this are Tor functors, the derived functors of the tensor product of modules over a ring. Besides playing a central role within algebraic topology theorems such as the K\"{u}nneth Theorem and Coefficient Theorem, Tor functors can also be used to calculate the homology of groups, Lie algebras, and associative algebras. Within the context of the classical Ravenel-Wilson Hopf ring method, the identification of the $E^2$-page of the bar spectral sequence with Tor allows for the computations $\text{Tor}^{E[x]}(\mF_p, \mF_p) \simeq \Gamma [s x] $ and $\text{Tor}^{T[x]} (\mF_p, \mF_p) \simeq E [s x] \otimes \Gamma[\phi x],$ where $sx$ is the suspension of $x,$ $\phi x$ is the transpotent, and $T[x]$ is the truncated polynomial ring $\mF_p [ x] / (x^p)$, to be used inductively in the calculations of the mod p homology of Eilenberg-MacLane spaces \cite{Wilson1982} and the Morava K-theory of Eilenberg-MacLane spaces \cite{RavanelWilson1980}. 

In the $C_2$-equivariant setting, the $RO(C_2)$-graded homology of each signed delooping, $K_{V + \sigma},$ of an equivariant Eilenberg-MacLane space, $K_{V},$ also independently arises as the result of a $C_2$-equivariant twisted Tor computation. Thus under favorable circumstances, we believe it should be possible to formulate a twisted bar spectral sequence with $E^2$-page a twisted Tor functor arising as a derived functor of the twisted product of $H \umF_2$-modules and use this to compute the $E^2$-page. However, we have not yet constructed such a spectral sequence. 

Additionally, twisted Tor calculations are not yet well understood, with a complete lack of known examples. Theorems \ref{Thm:signed}, \ref{Thm:HKVgen}, and \ref{Thm:sigmaPlusl} provide a countably infinite number of initial examples, which in turn lend insight on how such calculations might proceed in general. We discuss how the homology $H_\star K_{V + \sigma}$ arises as a result of twisted Tor and give evidence for $Tor_{tw}^{E[x]} (H_\star, H_\star) \simeq E[\sigma x] \otimes \Gamma[\cN_{e}^{C_2}(x)]$, where $\sigma x$ is the signed suspension of $x$ and $\cN_e^{C_2}$ is the norm, under favorable circumstances in Section \ref{sec:futurdir}.

\subsection{Statement of Theorems}\label{subsec:stateThms}

We state our main results. Recall that $H \umF_2$ has distinguished elements $a \in H {\umF_2}_{\{- \sigma\}}$ and $u \in {H \umF_2}_{\{ 1 - \sigma \} }$. 

To describe our answer for $H_\star K_{* \sigma}$, we need notation for $H_\star K_\sigma.$ Let 
\begin{align*}
	e_\sigma \in H_\sigma K_\sigma, && \bar{\alpha_i} \in H_{\rho i} K_\sigma, && (i \geq 0).
\end{align*}
Then the homology, $H_\star K_\sigma,$ is exterior on generators 
\begin{align*}
	e_\sigma, && \bar{\alpha}_{(i)} = \bar{\alpha}_{2^i} && (i \geq 0)
\end{align*}
with coproduct 
\begin{align*}
	\psi (e_\sigma) & = 1 \otimes e_\sigma + e_\sigma \otimes 1 + a (e_\sigma \otimes e_\sigma) \\
	\psi( \bar{\alpha}_{n} ) & = \sum_{i = 0}^n \bar{\alpha}_{n - i} \otimes \bar{\alpha}_i + \sum_{i = 0}^{n-1} u (e_\sigma \bar{\alpha}_{n - 1 - i} \otimes e_\sigma \bar{\alpha}_{i}). \\
\end{align*}

For finite sequences 
$$J = (j_\sigma, j_0, j_1, \cdots) \qquad j_k \geq 0,$$
define
$$(e_\sigma \bar{\alpha})^J = e_\sigma^{\circ j_\sigma} \circ \bar{\alpha}_{(0)}^{\circ j_0} \circ \bar{\alpha}_{(1)}^{\circ j_1} \cdots$$
where the $\circ$-product comes from the pairing $\circ: K_V \wedge K_W \to K_{V + W}.$
\begin{thm*}[Theorem \ref{Thm:signed}]  Then
	$$H_\star K_{*\sigma} \cong \otimes_J E[(e_\sigma \bar{\alpha})^J]$$
	As an algebra where the tensor product is over all J and the coproduct follows by Hopf ring properties from the $\bar{\alpha}$'s. 
\end{thm*}
\noindent Interestingly, this answer mirrors the classical non-equivariant answer at the prime 2 \cite{RavanelWilson1980}.

From there, we use the $RO(C_2)$-graded bar spectral sequence to inductively compute $H_\star K_{i \sigma + j}$ from $H_\star K_{i \sigma}$ and show
\begin{thm*}[Theorem \ref{Thm:HKVgen}]
	The $RO(C_2)$-graded homology of $K_V,$ where $\sigma + 1 \subset V,$ is exterior on generators given by the cycles on the $E^2$-page of the $RO(C_2)$-graded spectral sequence computing $H_\star B K_{V - 1}$.
\end{thm*}

For the spaces $K_{\sigma + *},$ we name all homology generators in terms of the Hopf ring structure. To describe these rings, we need notation for $H_\star K_1,$ $H_\star K_2,$ and $H_\star K_\rho.$ Let
$$	e_1 \in H_1 K_1, \qquad \qquad \alpha_i \in H_{2i} K_1, \qquad \qquad \beta_i \in H_{2i} C P^\infty \qquad \qquad i \geq 0. $$
This gives generators
$$e_1, \qquad \alpha_{(i)} = \alpha_{p^i} \qquad \beta_{(i)} = \beta_{p^i},$$
of $H_\star K_1$ and $H_\star K_2$ with coproducts
$$\psi(\alpha_n) = \displaystyle{ \sum_{i = 0}^n }\alpha_{n - i } \otimes \alpha_i, \qquad \psi(\beta_n) = \displaystyle{ \sum_{i = 0}^n \beta_{n - i} \otimes \beta_i } .$$
Also let 
$$\bar{\beta}_{i} \in H_{\rho i} K(\umZ, \rho), \, \qquad (i \geq 0).$$
This gives additional generators,
$$\bar{\beta}_{(i)} = \bar{\beta}_{2^i} \quad (i \geq 0), $$
of $H_\star K_\rho$ with coproduct  
$$ \psi( \bar{\beta}_{n} ) = \sum_{i = 0}^n \bar{\beta}_{n - i} \otimes \bar{\beta}_i.$$
Then for finite sequences 
$$ I = (i_1, i_2, \cdots, i_k ), \qquad 0 \leq i_1 < i_2 < \cdots, $$
$$ W = (w_1, w_2, \cdots, w_q), \qquad 0 \leq w_1 < w_2 < \cdots,$$ 
$$	J = (j_{-1}, j_0, j_1, \cdots, j_\ell ), \qquad \text{where } j_{-1} \in \{0,1\} \text{ and all other } j_n \geq 0, $$
and 
$$ Y = (y_{-1}, y_0, y_1, \cdots, y_r ), \qquad \text{where } y_{-1} \in \{0,1\} \text{ and all other } y_n \geq 0, $$
define
$$(e_1 \alpha \beta)^{I,J} = e_1^{\circ j_{-1}} \circ \alpha_{(i_1)} \circ \alpha_{(i_2)} \circ \cdots \circ \alpha_{(i_k)} \circ \beta_{(0)}^{\circ j_0} \circ \beta_{(1)}^{\circ j_1} \circ \cdots \circ \beta_{(\ell)}^{\circ j_\ell},$$
$$(e_1 \alpha \beta)^{W,Y} = e_1^{\circ y_{-1}} \circ \alpha_{(w_1)} \circ \alpha_{(w_2)} \circ \cdots \circ \alpha_{(w_q)} \circ \beta_{(0)}^{\circ y_0} \circ \beta_{(1)}^{\circ y_1} \circ \cdots \circ \beta_{(r)}^{\circ j_r},$$
$$ \vert I \vert = k, \qquad \vert W \vert = q \qquad \vert \vert J \vert \vert = \Sigma j_n, \qquad \text {and } \vert \vert Y \vert \vert = \Sigma y_n.$$

\begin{thm*}[Theorem \ref{Thm:sigmaPlusl}] We have
	$$ H_\star K_{\sigma + i} \cong E[(e_1 \alpha \beta)^{I,J} \circ \bar{\alpha}_{(m)}, (e_1 \alpha \beta)^{W, Y} \circ \bar{\beta}_{(t)}] $$
	where $m > i_k$ and $m \geq \ell,$ $t > w_q$ and $t \geq y_r,$ $\vert I \vert + 2 \vert \vert J \vert \vert = i$ and $\vert W \vert + 2 \vert \vert Y \vert \vert = i - 1,$  and the coproduct follows by Hopf ring properties from the $\alpha_{(i)}$'s, $\beta_{(i)}$'s, $\bar{\alpha}_{(i)}$'s and $\bar{\beta}_{(i)}$'s. 
\end{thm*}
We observe that this equivariant answer mirrors the classical non-equivariant answer for odd primes \cite{RavanelWilson1980}. For the reader's convenience, we explicitly write some low dimensional instances of the theorem. In particular, we have
$$H_\star K_\rho \cong E[ e_1 \circ \bar{\alpha}_{(i)}, \alpha_{(i_1)} \circ \bar{\alpha}_{(i_2)}, \bar{\beta}_{(i)}] $$
and
$$H_\star K_{\sigma + 2} \cong E[e_1 \circ \alpha_{(i_1)} \circ \bar{\alpha}_{(i_2)}, \alpha_{(i_1)} \circ \alpha_{(i_2)} \circ \bar{\alpha}_{(i_3)}, e_1 \circ \bar{\beta}_{(i_1)}, \beta_{(j_1)} \circ \bar{\alpha}_{(j_2)}, \alpha_{(i_1)} \circ \bar{\beta}_{(i_2)} ] $$ where $i_1 < i_2,$ $j_1 \leq j_2;$ and the coproduct follows by Hopf ring properties from the $\alpha_{(i)}$'s, $\beta_{(i)}$'s, $\bar{\alpha}_{(i)}$'s and $\bar{\beta}_{(i)}$'s.

Having computed the homology of the $C_2$-equivariant Eilenberg-MacLane spaces $K_V,$ we turn to using the results to investigate the twisted bar spectral sequence arising from the twisted bar construction. Unlike the non-equivariant bar spectral sequence, the twisted bar spectral sequence $E^2$ page lacks an explicit homological description. This makes computations difficult in general. However, for the spaces $B^\sigma \umF_2 \simeq K_\sigma \simeq \mR P_{tw}^\infty,$ $B^\sigma S^1 \simeq K(\umZ, \rho) \simeq \mC P_{tw}^\infty,$ and $B^\sigma S^\sigma \simeq K(\umZ, 2\sigma),$ there is a gap in the spectral sequence forcing all differentials $d^r$ for $r>1$ to be zero. Further for these spaces, if there were a non-zero $d^1$ differential, we would end up killing a known generator of the underlying non-equivariant integer graded homology and arrive at a contradiction. Thus we can calculate the additive $RO(C_2)$-graded homology of these spaces completely. The multiplicative structure can also be deduced from the twisted bar spectral sequence. 

\begin{exs*}[Example \ref{twistbssExs}] We have
	\begin{align*} 
		H_\star \mathbb{R} P_{tw}^\infty & = E[e_\sigma, \bar{\alpha}_{(0)}, \bar{\alpha}_{(1)}, \cdots ] = E[e_\sigma] \otimes \Gamma [ \bar{\alpha}_{(0)} ], \, \vert e_\sigma \vert = \sigma, \, \vert \bar{\alpha}_{(i)} \vert = \rho 2^i, \\
		H_\star \mathbb{C} P_{tw}^\infty & = E[ \bar{\beta}_{(0)}, \bar{\beta}_{(1)}, \cdots ] = \Gamma [ e_\rho ] \textrm{ where } \vert \bar{\beta}_{(i)} \vert = \rho 2^i. \\ 
	\end{align*}
\end{exs*}

\begin{thm*}[Theorem \ref{Thm:2sigma}] We have
	$$H_\star K(\underline{\mathbb{Z}}, 2 \sigma) = E[ e_{2\sigma} ] \otimes \Gamma[ \bar{x}_{(0)}] \textrm{ where } \vert e_{2\sigma} \vert = 2\sigma, \, \vert \bar{x}_{(0)} \vert = 2 \rho.$$
	
\end{thm*}

In forthcoming work, we will use the homology of $H_\star K_V$ to deduce differentials in the twisted bar spectral sequence. The beginning stages of this work are described in Section \ref{sec:btwbssComps}. 

\subsection{Paper structure}\label{subsec:paperStruc}

This paper has two primary aims: {(1)} extending Ravenel-Wilson Hopf ring techniques \cite{RavenelWilson1977, RavanelWilson1980, Wilson1982} to $C_2$-equivariant homotopy theory and {(2)} computing the $RO(C_2)$-graded homology of $C_2$-equivariant Eilenberg-MacLane spaces associated to the constant Mackey functor $\umF_2.$ These topics are investigated in several sections.

The first section consists of an introduction providing context for the main results, a description of the paper structure, and a list of notational conventions.

The second section recalls classical Ravenel-Wilson Hopf ring methods. 

The third section recollects material from equivariant homotopy theory necessary for understanding our proof and computations. 

The fourth section details the bar and twisted bar constructions, which are trivial and sign representation delooping functors respectively. 

The fifth section applies the preliminaries of the previous sections to study multiplicative structures on $C_2$-equivariant Eilenberg-MacLane spaces. This section contains some primary extensions of Ravenel-Wilson Hopf ring methods to $C_2$-equivariant homotopy theory (Theorem \ref{thm:5.1}). It also contains our calculation of the $RO(C_2)$-graded homology of many $C_2$-equivariant Eilenberg-MacLane spaces $K_V$ associated to the constant Mackey functor $\mF_2$ (Theorems \ref{Thm:signed}, \ref{Thm:HKVgen}, and \ref{Thm:sigmaPlusl}).

The sixth section details a number of computations and observations regarding the $RO(C_2)$-graded bar and twisted bar spectral sequences. The examples we provide should be a useful stepping stone towards further computations. 

The seventh section describes a few questions of immediate interest given the results of this paper.

\subsection{Notational conventions}\label{subsec:noteConv} \hfil

\noindent The asterisk $*$ denotes integer grading.

\noindent The star $\star$ to denote representation grading. 

\begin{hangingpar}
	\setlength{\hangingindent}{\parindent}
	 By the classical or nonequivariant Eilenberg-MacLane space $K_n,$ we mean the classical nonequivariant Eilenberg-MacLane space $K_n = K(\mF_p, n),$ where $p$ is prime.   
\end{hangingpar}

\begin{hangingpar}
	\setlength{\hangingindent}{\parindent}
	 $C_2$ is the cyclic group of order two with $C_2 = \langle \gamma \rangle.$
\end{hangingpar}

\begin{hangingpar}
	\setlength{\hangingindent}{\parindent}
	 $\sigma$ denotes the one dimensional sign representation of $C_2.$
\end{hangingpar}

\begin{hangingpar}
	\setlength{\hangingindent}{\parindent}
	 $\rho$ is the regular representation of $C_2.$
\end{hangingpar}

\begin{hangingpar}
	\setlength{\hangingindent}{\parindent}
	 $S^V$ is the one point compactification of a finite dimensional real representation $V$ where the point at infinity is given a trivial group action and taken as the base point.
\end{hangingpar}

\begin{hangingpar}
	\setlength{\hangingindent}{\parindent}
	 $\Sigma^V (-) = S^V \wedge -.$
\end{hangingpar}

\begin{hangingpar}
	\setlength{\hangingindent}{\parindent}
	 $\Omega^V(-)$ is the space of continuous based maps $\text{Map}_* (S^V, -)$ where the group action is given by conjugation.
\end{hangingpar}

\begin{hangingpar}
	\setlength{\hangingindent}{\parindent}
	 $\mathcal{S}$ is the category of spectra.
\end{hangingpar}

\begin{hangingpar}
	\setlength{\hangingindent}{\parindent}
	 $\mathcal{S}^G$ is the category of $G$-spectra indexed on a complete universe.
\end{hangingpar}

\subsection*{Acknowledgments}

The author would like to thank Mark Behrens for many helpful conversations and guidance throughout this project. The author is also grateful for conversations with Prasit Bhattacharya, Bertrand Guillou, Mike Hill, and Carissa Slone. The author was partially supported by the National Science Foundation under Grant No. DMS-1547292, and would also like to thank the Max Planck Institute for Mathematics for its hospitality and financial support. 

\section{Classical Ravenel-Wilson Hopf ring methods}\label{sec:RWhopf}

Classically, one place Hopf rings arise in homotopy theory is in the study of $\Omega$-spectra. Consider an $\Omega$-spectrum 
$$G = \{G_k\}$$
and a multiplicative homology theory $E_*(-)$ with a K\"{u}nneth isomorphism for the spaces $G_k$.
The $\Omega$-spectrum $G$ represents a generalized cohomology theory with
$$G^*X \simeq [X, G_*].$$ 
Since $G^k X$ is an abelian group, $G_k$ must be a homotopy commutative $H$-space (in fact $G_k$ is an infinite loop space). This $H$-space structure 
$$*: G_k \times G_k \to G_k$$
gives rise to a product in homology
$$*: E_* G_k \otimes E_* G_k \cong E_*(G_k \times G_k) \to E_* G_k$$ and the K\"{u}nneth isomorphism implies the homology is in fact a Hopf algebra. 

If $G$ is a ring spectrum, then $G^*X$ is a graded ring and the graded abelian group object $G_*$ becomes a graded ring object in the homotopy category. The multiplication
$$G^kX \times G^n X \to G^{k + n} X$$
has a corresponding multiplication in $G_*:$
$$\circ: G_k \times G_n \to G_{k +n}$$
and applying $E_*(-)$ we have 
$$\circ : E_* G_k \otimes_{E_*} E_* G_n \to E_* G_{k + n}$$
turning $E_*G$ into a graded ring object in the category of coalgebras. 

As a ring, $E_*G$ has a distributive law 
\begin{equation} \label{eq:distlaw}
	x \circ (y * z) = \sum \pm (x' \circ y) * (x" \circ z) \text{ where } \psi(x) = \sum x' \otimes x"
\end{equation}
coming from the distributive law in $G^* X.$ 

Ravenel and Wilson pursued the idea that these two products could be used to construct many elements in homology from just a few. They successfully applied this approach to compute the Hopf ring for complex cobordism \cite{RavenelWilson1977}, the Morava K-theory of nonequivariant Eilenberg-MacLane spaces \cite{RavanelWilson1980}, and the mod $p$ homology of classical Eilenberg-MacLane spaces \cite{Wilson1982}. 

In the case of classical Eilenberg-MacLane spaces, the Eilenberg-MacLane spectrum
$$H \mF_p = \{K( \mF_p, n)\} = \{K_n\},$$
is a ring spectrum with $\Omega K_{n + 1} \simeq K_n$. Further, $H_*(-) := H_*(-; \mF_p)$, ordinary homology with mod p coefficients, has a K\"{u}nneth isomorphism and thus the homology $H_*K_*$ has the structure of a Hopf ring. 

A key computational insight of Ravenel and Wilson was that the bar spectral sequence 
$$E_{*, *}^2 \simeq Tor_{*,*}^{E_*G_k} (E_*, E_*) \Rightarrow E_* G_{k + 1} $$ is in fact a spectral sequence of Hopf algebras. 
The additional structure of the $\circ$ multiplication in the bar spectral sequence meant that they could inductively deduce the homology of Eilenberg-MacLane spaces using standard homological algebra. Starting with elements in $H_* K_1$ and $H_* \mC P^\infty$ and identifying circle products in the bar spectral sequence, Ravenel and Wilson computed the Hopf ring associated to the mod p Eilenberg-MacLane spectrum \cite{Wilson1982}.

To describe their answer, let
$$	e_1 \in H_1 K_1, \qquad \alpha_i \in H_{2i} K_1, \qquad \beta_i \in H_{2i} C P^\infty \qquad i \geq 0. $$
The generators are
$$e_1, \qquad \alpha_{(i)} = \alpha_{p^i} \qquad \beta_{(i)} = \beta_{p^i} $$
with coproduct 
$$\psi(\alpha_n) = \displaystyle{ \sum_{i = 0}^n }\alpha_{n - i } \otimes \alpha_i, \qquad \psi(\beta_n) = \displaystyle{ \sum_{i = 0}^n \beta_{n - i} \otimes \beta_i } .$$
For finite sequences,
$$ I = (i_1, i_2, \cdots ), \qquad 0 \leq i_1 < i_2 < \cdots, $$
and 
$$	J = (j_0, j_1, \cdots ), \qquad j_k \geq 0, $$
define
$$\alpha_I = \alpha_{(i_1)} \circ \alpha_{(i_2)} \circ \cdots, \\
\beta^J = \beta_{(0)}^{\circ j_0} \circ \beta_{(1)}^{\circ j_1} \circ \cdots, $$
and let	$T(x)$ denote the truncated polynomial algebra $\mF_p[x] / (x^p) $. 

\begin{thm}[Ravenel-Wilson \cite{Wilson1982}] \label{Thm:RWpodd} We have
	$$H_* K_* \simeq \otimes_{I, J} E(e_1 \circ \alpha_I \circ \beta^J) \otimes_{I, J} T(\alpha_I \circ \beta^J)$$
	as an algebra where the tensor product is over all I, J and the coproduct follows by Hopf ring properties from the $\alpha$'s and $\beta$'s. 
\end{thm}

When the prime $p = 2,$ there are additional relations $e_1 \circ e_1 = \beta_{(0)}$ and $\alpha_{(i - 1)} \circ \alpha_{(i - 1)} = \beta_{(i)}.$ In this case, the theorem can be stated using only circle products of generators of $\mR P^\infty.$ 

For finite sequences 
$$I = (i_{(-1)}, i_0, i_1, i_2, \cdots ), \qquad \qquad i_k \geq 0  $$
define
$$(e_1 \alpha)^I = e_1^{\circ i_{(-1)}} \circ \alpha_{(0)}^{\circ i_0} \circ \alpha_{(1)}^{\circ i_1} \circ \cdots.$$

\begin{thm}[Ravenel-Wilson \cite{Wilson1982}] \label{Thm:RWp2} Then
	$$H_* K_{n} \cong \otimes_I E[(e_1 \alpha)^I] $$
	where $\sum i_k = n,$ and considering all spaces at once, 
	$$H_* K_* \simeq \otimes_{I} E[(e_1 \alpha)^I]$$
	as an algebra where the tensor product is over all I and the coproduct follows by Hopf ring properties from the $\alpha$'s. 
\end{thm}

Ravenel and Wilson also show that homology suspending $\beta_{(i)}$ to define
$$\xi_i \in H_{2(p^i - 1)} H,$$
and $\alpha_{(i)}$ to define 
$$\tau_i \in H_{2p^i - 1}H, $$
Theorem \ref{Thm:RWpodd} then implies that stably,
$$H_*H \simeq E[\tau_0, \tau_1, \cdots] \otimes P[\xi_1, \xi_2, \cdots ].$$

\section{Equivariant preliminaries}\label{sec:equivarPrelims}

We set notation and recall equivariant foundations. Throughout, the group $G = C_2$.

Given an orthogonal real $G$-representation $V,$ $S^V$ denotes the representation sphere given by the one-point compactification of $V.$ For a $p$-dimensional real $C_2$ - representation $V,$ we write
$$V \cong \mathbb{R}^{(p-q,0)} \oplus \mathbb{R}^{(q,q)} $$
where $\mR^{(1,0)}$ is the trivial 1-dimensional real representation of $C_2$ and $\mR^{(1,1)}$ is the sign representation. We allow $p$ and $q$ to be integers, so $V$ may be a virtual representation. The integer $p$ is called the topological dimension while $q$ is the weight or twisted dimension of $V \cong \mR^{(p,q)}.$

The $V$-th graded component of the ordinary $RO(C_2)$-graded Bredon equivariant homology of a $C_2$-space $X$ with coefficients in the constant Mackey functor $\umF_2$ is denoted $H_V^{C_2}(X ; \umF_2) = H_{p,q} (X; \umF_2).$ To consider all representations at once we write $H_\star (X),$ and when working nonequivariantly $H_* (X^e)$ denotes the singular homology of the underlying topological space with $\mF_2$ coefficients.

It is often convenient to plot the bigraded homology in the plane. Our plots have the topological dimension p on the horizontal axis and the weight q on the vertical axis. 

The homology of a point with coefficients in the constant Mackey functor $\umF_2$, is the bi-graded ring 
$$H_\star (pt, \umF_2) = \mF_2[a, u] \oplus \frac{\mF_2[a,u]}{(a^\infty, u^\infty)} \{ \theta \}$$
where $\vert a \vert = -\sigma,$ $\vert u \vert = 1 - \sigma,$ and $\vert \theta \vert = 2\sigma - 2.$ A bi-graded plot of $H_\star (pt, \umF_2)$ appears in Figure \ref{fig:HomPt}. The image on the left is more detailed with each lattice point within the two cones representing a copy of $\mF_2.$ The image on the right is a more succinct representation and appears in figures illustrating our spectral sequence computations.  

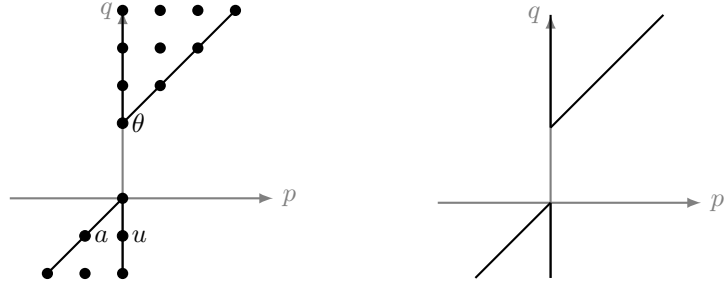
\begin{figure}[ht]	
	\begin{center}
		\begin{tikzpicture}[scale=0.5]
			
			\coordinate (Origin)   at (0,0);
			\coordinate (XAxisMin) at (-3,0);
			\coordinate (XAxisMax) at (4,0);
			\coordinate (YAxisMin) at (0,-2);
			\coordinate (YAxisMax) at (0,5);
			\draw [thick, gray,-latex] (XAxisMin) -- (XAxisMax) node[anchor=west] {$p$};
			\draw [thick, gray,-latex] (YAxisMin) -- (YAxisMax) node[anchor=east] {$q$};
			
			\draw [thick] (0,2) -- (0,5);
			\draw [thick] (0,2) -- (3,5);
			\draw [thick] (0,0) -- (0,-2);
			\draw [thick] (0,0) -- (-2,-2);
			
			\node at (0,0) [circle,fill,inner sep=1.5pt]{};
			\filldraw[black] (0,-1) circle (4pt) node[anchor=west] {$u$};
			\node at (0,-2) [circle,fill,inner sep=1.5pt]{};
			\filldraw[black] (-1,-1) circle (4pt) node[anchor=west] {$a$};
			\node at (-1,-2) [circle,fill,inner sep=1.5pt]{};
			\node at (-2,-2) [circle,fill,inner sep=1.5pt]{};
			
			\filldraw[black] (0,2) circle (4pt) node[anchor=west] {$\theta$};
			\node at (0,3) [circle,fill,inner sep=1.5pt]{};
			\node at (0,4) [circle,fill,inner sep=1.5pt]{};
			\node at (0,5) [circle,fill,inner sep=1.5pt]{};
			\node at (1,3) [circle,fill,inner sep=1.5pt]{};
			\node at (1,4) [circle,fill,inner sep=1.5pt]{};
			\node at (1,5) [circle,fill,inner sep=1.5pt]{};
			\node at (2,4) [circle,fill,inner sep=1.5pt]{};
			\node at (2,5) [circle,fill,inner sep=1.5pt]{};
			\node at (3,5) [circle,fill,inner sep=1.5pt]{};
			
		\end{tikzpicture}
		\begin{tikzpicture}[scale=0.5]
			\node at (-6,0) {$\,$};
			
			\coordinate (Origin)   at (0,0);
			\coordinate (XAxisMin) at (-3,0);
			\coordinate (XAxisMax) at (4,0);
			\coordinate (YAxisMin) at (0,-2);
			\coordinate (YAxisMax) at (0,5);
			\draw [thick, gray,-latex] (XAxisMin) -- (XAxisMax) node[anchor=west] {$p$};
			\draw [thick, gray,-latex] (YAxisMin) -- (YAxisMax) node[anchor=east] {$q$};
			
			\draw [thick] (0,2) -- (0,5);
			\draw [thick] (0,2) -- (3,5);
			\draw [thick] (0,0) -- (0,-2);
			\draw [thick] (0,0) -- (-2,-2);
			
			\filldraw[white] (1,5) circle (2pt) node[anchor=west] {};
		\end{tikzpicture}
	\end{center}
	\caption{$H_\star(pt, \umF_2)$ with axis gradings determined by $V \simeq \mathbb{R}^{p-q} \oplus \mathbb{R}^{q \sigma}.$ }
	\label{fig:HomPt}
\end{figure}

The genuine equivariant Eilenberg-MacLane spectrum representing $H_\star(-)$ is $H\umF_2,$ the Eilenberg-MacLane spectrum for the $C_2$ constant Mackey functor $\umF_2.$ It has underlying nonequivariant spectrum $H \mF_2.$ We denote the spaces of $H \umF_2$ 
$$H \underline{\mathbb{F}}_2 = \{K(\underline{\mathbb{F}}_2, V)\}_{V \cong k \sigma + l} = \{K_V\}_{V \cong k \sigma + l}.$$
Analogously to the nonequivariant case, $H \underline{\mathbb{F}}_2$ is characterized up to $C_2$-equivariant homotopy by $H^V(X; \underline{\mathbb{F}}_2) = [X, K_V]$ naturally for all $C_2$-spaces X.

We recall a computational lemma due to Behrens and Wilson \cite{BehrensWilson2018}, which allows us to check whether a set of elements in the $RO(C_2)$-homology in fact forms a free basis for $H_\star(X),$ greatly simplifying our computations. To state this lemma, we first define two homomorphisms, $\Phi^e$ and $\Phi^{C_2}.$ Let $Ca$ be the cofiber of the Euler class $a \in \pi_{-\sigma}^{C_2} S$ given geometrically by the inclusion 
$$S^0 \hookrightarrow S^\sigma .$$ Applying $\pi_V^{C_2}$ to the map
$$H \wedge X \to H \wedge X \wedge Ca, $$
we get a homomorphism 
$$\Phi^e: H_V(X) \to H_{\vert V \vert }(X^e).$$
Taking geometric fixed points of a map
$$S^V \to H \wedge X $$
gives a map 
$$S^{V^{C_2}} \to H^{\Phi C_2} \wedge X^{\Phi C_2}.$$
Using the equivalence $H_*^\Phi X \simeq H_*(X^{\Phi C_2}) [a^{-1} u] $ coming from $H^{\Phi C_2} \simeq \displaystyle{\bigvee_{i \geq 0} } \Sigma^i H \mF_2 $ and passing to the quotient by the ideal generated by $a^{-1} u$ gives the homomorphism
$$\Phi^{C_2}: H_V(X) \to H_{\vert V^{C_2} \vert} (X^{\Phi C_2}).$$

\begin{lemma}[Behrens-Wilson \cite{BehrensWilson2018}] \label{lem:BW}
	Suppose $X \in Sp^{C_2}$ and $\{b_i\}$ is a set of elements of $H_\star (X)$ such that 
	\begin{itemize}
		\item[(1)] $\{\Phi^e (b_i) \}$ is a basis of $H_*(X^e)$ and
		\item[(2)] $\{\Phi^{C_2} (b_i)\}$ is a basis of $H_*(X^{\Phi C_2}),$ 
	\end{itemize}
	then $H_\star (X)$ is free over $H_\star$ and $\{b_i\}$ is a basis. 
\end{lemma}

We use the following notation for $H^\star K_\sigma.$

\begin{theorem}[Hu-Kri\cite{HuKriz2001}]
	$H^\star (\mR P^\infty_{tw}) = H^\star(pt)[\alpha,\beta]/(\alpha^2 = a \alpha + u \beta)$ where $\vert \alpha \vert = \sigma,$ $\vert \beta \vert = \rho,$ $\vert a \vert = \sigma,$ and $ \vert u \vert = 1 - \sigma.$ 
\end{theorem}
Since this cohomology is free, the homology $H_\star K_\sigma$ immediately follows. In our notation we have elements
\begin{align*}
	e_\sigma \in H_\sigma K_\sigma, && \bar{\alpha_i} \in H_{\rho i} K_\sigma, && (i \geq 0).
\end{align*}
The generators are 
\begin{align*}
	e_\sigma, && \bar{\alpha}_{(i)} = \bar{\alpha}_{2^i} && (i \geq 0)
\end{align*}
with coproduct  
\begin{align*}
	\psi (e_\sigma) & = 1 \otimes e_\sigma + e_\sigma \otimes 1 + a (e_\sigma \otimes e_\sigma) \\
	\psi( \bar{\alpha}_{n} ) & = \sum_{i = 0}^n \bar{\alpha}_{n - i} \otimes \bar{\alpha}_i + \sum_{i = 0}^{n-1} u (e_\sigma \bar{\alpha}_{n - 1 - i} \otimes e_\sigma \bar{\alpha}_{i}). \\
\end{align*}
and ring structure $H_\star K_\sigma \simeq E[e_\sigma, \bar{\alpha}_{(i)}]$ which can be deduced from the twisted bar spectral sequence computing $H_\star B^\sigma \mathbb{F}_2 \cong H_\star \mR P_{tw}^\infty.$

We also require notation for $H_\star K(\umZ, \rho).$ This can be deduced by applying the $RO(C_2)$-graded bar spectral sequence to $S^\sigma.$ Let 
$$\bar{\beta}_{i} \in H_{\rho i} K(\umZ, \rho) \qquad (i \geq 0).$$
The generators are 
$$
\bar{\beta}_{(i)} = \bar{\beta}_{2^i} \qquad (i \geq 0) $$
with coproduct  
$$ \psi( \bar{\beta}_{n} ) = \sum_{i = 0}^n \bar{\beta}_{n - i} \otimes \bar{\beta}_i$$
and ring structure
$$H_\star K(\umZ, \rho) \simeq E[\beta_{(i)}].$$

\subsection{The fixed point spaces of $C_2$-equivariant Eilenberg-MacLane spaces} \label{subsec:fixedptspaces}

It is useful to understand the $C_2$ fixed points of the $C_2$-equivariant Eilenberg-MacLane spaces $K_V$ in applications of the Behrens-Wilson computational lemma. We state a relevant proposition due to Caruso. 

\begin{proposition}[Caruso \cite{Caruso1999}] Let $G = C_p$ and $V$ be an $n$-dimensional fixed point free virtual representation of $G$ with $n > 0$ and $m$ an integer. Then 
	$$K(\umF_p, m + V)^{C_p} \simeq K(\mF_p, m) \times \cdots \times K(\mF_p, m +n) .$$
\end{proposition}

\subsection{Notation for the underlying nonequivariant homology of $K_V^{C_2}$} \label{subsec:underlyinghom}

To use the Behrens-Wilson lemma, we also need to understand the homology of the fixed point spaces. Applying Theorem \ref{Thm:RWp2} to the nonequivariant homology of $(K_{n \sigma})^{C_2}$ gives
$$H_* (K_{n \sigma}^{C_2}) \simeq E[e_0, a_{(i_1)}, a_{(i_1)} \circ a_{(i_2)}, \cdots, a_{(i_1)} \circ \cdots \circ a_{(i_n)} ]$$
where $0 \leq i_1 \leq i_2 \leq \cdots \leq i_n,$ $\vert e_0 \vert = 0,$ where $\vert a_{(i)} \vert = 2^i.$  

\section{Bar and twisted bar constructions}\label{sec:barConsts}

A first task in implementing the Ravenel-Wilson Hopf ring approach is to generalize the bar spectral sequence to the $C_2$-equivariant case. In the classical story, the bar spectral sequence is used to inductively compute the homology of $K_n \simeq B K_{n - 1} $ from $H_* K_{n-1}.$ In the $C_2$-equivariant setting, our spaces $K_V$ are bi-graded on the trivial and sign representations of $C_2$. Due to this new grading, we should now additionally compute the homology of $K_{V + \sigma}$ inductively from $H_\star K_{V}.$ In order to do so, we need a good model of $\sigma$-delooping. We begin by reviewing the classical bar construction which is a trivial representation delooping functor.  

\begin{construction}{\textbf{(The Classical Bar Construction)}} \newline
	For a topological monoid $A,$ the pointed space $B A$ is defined as a quotient
	$$B A = {\coprod_{n} \Delta ^n \times A ^{\times n} / \sim}$$
	where the relation $\sim$ is generated by 
	\begin{itemize}
		\item[(1)] $$(t_1, \cdots t_n, a_1, \cdots , a_n) \sim (t_1, \cdots, \hat{t_i}, \cdots, t_n, a_1, \cdots, \hat{a_i}, (a_i a_{i + 1}) \cdots, a_n)$$
		if $t_i = t_{i +1}$ or $x_i = *,$ 
		
		\item[(2)]for $i = n,$ delete the last coordinate if $t_n = 1$ or $a_n = *,$
		
		for $i = 0,$ delete the first coordinate if $t_0 = -1$ or $a_0 = *,$
		
		and $ \Delta^n$ denotes the topological simplex 
		$$\Delta^n = \{(t_1, t_2, \cdots, t_n) \in \mathbb{R}^{n} | -1 \leq t_1 \leq \cdots \leq t_n \leq 1 \}.$$
	\end{itemize}
\end{construction}

\begin{remark}
	We use the slightly nonstandard topological $n$-simplex 
	$$\Delta^n = \{(t_1, t_2, \cdots, t_n) \in \mathbb{R}^{n} | -1 \leq t_1 \leq \cdots \leq t_n \leq 1 \}$$ so that when we introduce a $C_2$ action, the simplex rotates around the origin. This makes writing down a model for the $H$-space structure on the $C_2$-equivariant Eilenberg-MacLane spaces $K_V$ more straightforward.
\end{remark}

Given a commutative monoid $A,$ we observe that $BA$ is also a commutative monoid via the pairing
$$* : B X \times B X \to B X$$ defined by 
\begin{align*}
	(t_1, \cdots, t_n,\, x_0, \cdots, x_n) *_\sigma (t_{n+1}, \cdots ,&  t_{n+m}, \, x_{n +1}, \cdots , x_{n + m}) \\
	& = (t_{\tau(1)}, \cdots, t_{\tau (n +k)} , x_{\tau(1)}, \cdots, x_{\tau(n + k)}),
\end{align*}
where $\tau$ is any element of the symmetric group on $n + k$ letters such
that $t_{\tau(i)} \leq t_{\tau(i + 1)}.$ This pairing was first described by Milgram \cite{Milgram1967}. 

\begin{definition}[\cite{Liu2020}]
	A $C_2$-space $A$ is a twisted monoid if it is a topological monoid in the non-equivariant sense with the product satisfying $\gamma (xy) = \gamma(y) \gamma(x)$ where $C_2 \simeq \langle \gamma \rangle.$ 
\end{definition}

\begin{construction}[\cite{Liu2020}]
	For any twisted monoid $A,$ construct $B_*^\sigma A$ in the same way as the non-equivariant bar construction, that is such that $B_n^\sigma A = \Delta^n \times A^n.$ However, define a $C_2$-action on $A^n$ by 
	$$\gamma (a_1, a_2, \cdots, a_n) = (\gamma a_n, \gamma a_{n-1}, \cdots, \gamma a_1).$$
	Then the $C_2$-actions commute with the face and degeneracy maps as $\gamma \circ s_i = s_{n-i} \circ \gamma$ and $\gamma \circ d_i = d_{n - i} \circ \gamma.$ Further define the $C_2$-action on each 
	$$\Delta^n = \{(t_1, t_2, \cdots, t_n) \in \mathbb{R}^{t + 1} | -1 \leq t_1 \leq \cdots \leq t_n \leq 1 \}.$$
	by $\gamma(t_1, t_2, \cdots, t_n) = (-t_n, -t_{n-1}, \cdots, -t_1) ).$
	Then define $B^\sigma A$ to be the geometric realization
	$$\displaystyle{\coprod  \Delta^n \times A^n / \sim}.$$
\end{construction} 

\begin{example}
	The space $B^\sigma K_0 \simeq \mR P^\infty_{tw}$ is the space of lines through the origin with conjugate action.   
\end{example}

We can inductively define an $H$-space pairing on $B^{l} B^{k \sigma} \mathbb{F}_2,$ similar to the one given by Milgram in the non-equivariant case. Define a mapping 
$$*_\sigma: B^\sigma X \times B^\sigma X \to B^\sigma X$$
by 
\begin{align*}
	(t_0, \cdots, t_n, \, x_0, \cdots, x_n) *_\sigma (t_{n+1}, \cdots ,&  t_{n+m}, \, x_{n +1}, \cdots , x_{n + m}) \\
	& = (t_{\tau(1)}, \cdots, t_{\tau (n +k)} , x_{\tau(1)}, \cdots, x_{\tau(n + k)}),
\end{align*}
where $\tau$ is any element of the symmetric group on $n + k$ letters such
that $t_{\tau(i)} \leq t_{\tau(i + 1)}.$ 
Then $*_\sigma$ is well defined, continuous, and $C_2$-equivariant. Going forward, we suppress the $\sigma$ notation in $*_\sigma,$ using only $*$ to denote the $H$-space pairing. The relevant $C_2$-action is deduced from context.  

\begin{proposition}[Liu \cite{Liu2020}]
	For any commutative monoid $A$ in the category of based $C_2$-spaces, the $V$-degree bar construction $B^V A$ is defined by applying the ordinary bar construction $l$ times and the twisted bar construction $m$ times for $V = l + m \sigma.$ There exists a natural map $A \to \Omega^V B^V A.$ When $A$ is $C_2$-connected, this map is a $C_2$-equivalence. 
\end{proposition}

\section{\texorpdfstring{Multiplicative structures on $C_2$-equivariant Eilenberg-MacLane spaces}{Multiplicative structures on equivariant Eilenberg-MacLane spaces}}\label{sec:multStrs}
	
	We describe multiplicative structures on $C_2$-equivariant Eilenberg-MacLane spaces, extending Ravenel and Wilson's description of similar structures on classical non-equivariant Eilenberg-MacLane spaces. We use our understanding of these structures to compute the $RO(C_2)$-graded homology of many $C_2$-equivariant Eilenberg-MacLane spaces $K_V$ associated to the constant Mackey functor $\mF_2.$ In particular, we compute the $RO(C_2)$-graded homology of all $C_2$-equivariant Eilenberg-MacLane spaces $K_{*\sigma}$ and $K_{\sigma + *}.$ 
	
	\subsection{Multiplicative structures on $K_V$} 

The $RO(C_2)$-graded cup product is induced by a map 
\begin{equation} \label{eq:5}
	\circ = \circ_{V,W} : K_V \wedge K_W \to K_{V + W}.
\end{equation}

We will construct $\circ_{V,W}$ explicitly within the framework of trivial and $\sigma$ - representation delooping given by $B$ and $B^\sigma$. We will also discuss how $\circ_{V,W}$ descends to a product on the fixed points. 

Given a real $C_2$ representation $V \cong l + k \sigma,$ the Eilenberg-MacLane space $K_V$ is a $V$-fold delooping of $\mathbb{F}_2$ and therefore can be constructed iteratively by taking $B^{l} B^{k \sigma} \mathbb{F}_2$ where $l,k$ are non-negative integers. The following construction  extends exposition by Ravenel and Wilson in their computation of the Morava $K$ theory of Eilenberg-MacLane spaces \cite{RavanelWilson1980}.

We construct the map \refeq{eq:5} inductively on $V$. Assuming $\circ_{V,W}$ has been defined, we define $\circ_{V+1, W}$ and $\circ_{V + \sigma, W}$ by replacing $K_{V + 1},$ $K_{V + W + 1},$ and $K_{V + \sigma},$ $K_{V + W + \sigma}$ with their bar and twisted bar constructions respectively. In both cases this is denoted as follows. There is a notationally suppressed $C_2$-action each case.

\begin{equation} \label{eq:6}
	\{ \coprod_{n} \Delta^n \times K_V^n / \sim \} \wedge K_W \to \{ \coprod_{n} \Delta^n \times K_{V + W} / \sim \}
\end{equation}

Let $t \in \Delta^n,$ $x = (x_0, \cdots, x_n) \in K_V,$ and $y \in K_W.$ The image of $x_i \wedge y \in K_V \wedge K_W$ under the map (\refeq{eq:5}) is denoted $x_i \circ y.$ We use the notation $x \circ y$ to mean $(x_0 \circ y, \cdots, x_n \circ y).$ Define (\refeq{eq:6}) by 

\begin{equation} \label{eq:7}
	\{(t, x) \} \circ y = \{(t, x \circ y)\}.
\end{equation}

\begin{theorem} \label{thm:5.1}
	The above construction is well defined and gives the cup product pairings
	$\circ: K_{V + 1} \wedge K_W \to K_{V + W + 1}$ and $\circ: K_{V + \sigma} \wedge K_W \to K_{V + W + \sigma}$. 
\end{theorem}

\begin{lemma} \label{lem:6.2}
	The map $\circ: K_0 \times K_V \to K_V$ is given by $(q) \circ x = x^{*q}$ where $q \in \mathbb{F}_2.$ 
\end{lemma}

\begin{proof}
	This map multiplies $\pi_V^{C_2} K_V \simeq \mathbb{F}_2$ by $q$ which is what $\circ$ should do restricted to $(q) \times K_V \simeq K_V.$ 
\end{proof}

\begin{proof}{\textit{of Theorem \ref{thm:5.1}}}
	This proof is a direct extension of the nonequivariant argument in Ravenel-Wilson \cite{RavanelWilson1980}. We prove our result by induction on $i$ in the $\sigma$ direction noting that the result also holds and is similar in the trivial representation direction (that is we assume the statement holds for $V$, show for $V+\sigma$). Assume we have proved Theorem 1 for $K_V \wedge K_W \to K_{V + W}$ with Lemma 1 beginning the induction. We need our construction to satisfy:
	\begin{align*}
		(z_1 * z_2) \circ y = (z_1 \circ y ) * (z_2 \circ y).
	\end{align*}
	For $i = 0,$ $z_i = q_i \in \mathbb{F}_2 = K_0.$ So, 
	$$(q_1 * q_2) \circ y = (q_1 + q_2) \circ y = y^{q_1 + q_2} = y^{*q_1} * y^{*q_2} = (q_1 \circ y) * (q_2 \circ y).$$ For $i > 0,$
		\begin{align*}
			[z_1 * z_2] \circ y & = [(t,x)*(t_{n+1}, \cdots, t_{n+k}; x_{n+1}, \cdots, x_{n+k})] \circ y  \\
			& = (t_{\tau(1), \cdots, t_{\tau(n + k)}; x_{\tau(1)}, \cdots, x_{\tau(n + k)}}) \circ y  \\
			& = (t_{\tau(1)}, \cdots, t_{\tau{(n + k)}}; x_{\tau(1)} \circ y, \cdots, x_{\tau(n + k)} \circ y) \\
			& = (t; x \circ y) ) * (t_{n + 1}, \cdots, t_{n + k}; x_{n + 1} \circ y, \cdots, x_{n + k} \circ y) \\
			& = (z_1 \circ y) * (z_2 \circ y) 
		\end{align*}
	where the second line is due to the definition of $*,$ the third is due to the induction hypothesis and \eqref{eq:7}, and the fourth is due to the definition of $*.$
\end{proof}

We must show \refeq{eq:7} gives well defined maps $K_{V + 1} \wedge K_W \to K_{V + W + 1}$ and $K_{V + \sigma} \wedge K_W \to K_{V + W + \sigma}.$ The relations in the (twisted) bar construction make this the case. We show the main case, leaving the others to the reader. Assume $0 \leq q < n$ with $t_q = t_{q + 1}$ or $x_q = *.$ Then

\begin{align*}
	(t, x) \circ y & = (t, x \circ y) \\
	& \sim (t_1, \cdots, \hat{t_q}, \cdots, t_n; x_1 \circ y, \cdots , (x_q \circ y)*(x_{q + 1} \circ y), \cdots, x_n \circ y ) \\
	& = (t_1, \cdots, \hat{t_q}, \cdots, t_n; x_1 \circ y, \cdots, (x_q * x_{q + 1}) \circ y, \cdots, x_n \circ y) \\
	& = (t_1, \cdots, \hat{t_q}, \cdots, t_n ; x_1, \cdots, x_q * x_{q + 1}, \cdots , x_n ) \circ y\\
\end{align*}
which is the necessary relation. That this map factors through the smash product is straightforward to verify using induction.

The remaining task is to show that this is the cup product pairing map. This follows by induction from the observation that $\circ$ commutes with (signed) suspension on the first factor since $B_1 K_V \simeq S^1 \wedge K_V$ and $B_1^\sigma K_V \simeq S^\sigma \wedge K_V,$ and following diagrams commute.

\[
\begin{tikzcd}
	S^1 \wedge K_V \wedge K_W \arrow{r}{} \arrow[swap]{d}{} & S^1 \wedge K_{V + W} \arrow{d}{} \\
	K_{V + 1} \wedge K_W \arrow{r}{} & K_{V + W + 1}
\end{tikzcd}
\]

\[
\begin{tikzcd}
	S^\sigma \wedge K_V \wedge K_W \arrow{r}{} \arrow[swap]{d}{} & S^\sigma \wedge K_{V + W} \arrow{d}{} \\
	K_{V + \sigma} \wedge K_W \arrow{r}{} & K_{V + W + \sigma}
\end{tikzcd}
\]

\subsection{Multiplicative structures on $K_V^{C_2}$}

We turn to understanding the $\circ$-product on the fixed points of the spaces $K_V.$ 	Notice $(B^\sigma A)^{C_2}$ consists of points of the form 
$$(t_1, \cdots, t_n, 0, -t_n, \cdots , -t_1, a_1, \cdots, a_n, a, \gamma(a_n), \cdots, \gamma(a_1) ) \in (B^\sigma A)^{[2n + 1]}$$
where $a \in A^{C_2}$ 
since for 
$$(t_1, \cdots, t_{m}, -t_{m}, \cdots, -t_1, a_1, \cdots, a_{m}, \gamma(a_{m}) \cdots, \gamma(a_1)) \in (B^\sigma A)^{[2m]},$$ 
there is a degeneracy map inducing an equivalence to
$$(t_1, \cdots, t_n, 0, -t_n, \cdots , -t_1, a_1, \cdots, a_n, *, \gamma(a_n), \cdots, \gamma(a_1) ) \in (B^\sigma A)^{[2n + 1]}.$$
Taking the fixed points in the construction of map (\refeq{eq:6}) we recover the classical nonequivariant $\circ$ product on the fixed point spaces.  

\subsection{Circle product generators for $H_\star K_{n \sigma}$}

Recall that $H \umF_2$ has generators $a \in H {\umF_2}_{\{- \sigma\}}$ and $u \in {H \umF_2}_{\{ 1 - \sigma \} }$. To describe our answer, we recall our notation for $H_\star K_\sigma$. Let 
\begin{align*}
e_\sigma \in H_\sigma K_\sigma, && \bar{\alpha_i} \in H_{\rho i} K_\sigma, && (i \geq 0).
\end{align*}
The homology, $H_\star K_\sigma,$ is exterior on generators 
\begin{align*}
e_\sigma, && \bar{\alpha}_{(i)} = \bar{\alpha}_{2^i} && (i \geq 0)
\end{align*}
with coproduct  
\begin{align*}
\psi (e_\sigma) & = 1 \otimes e_\sigma + e_\sigma \otimes 1 + a (e_\sigma \otimes e_\sigma) \\
\psi( \bar{\alpha}_{n} ) & = \sum_{i = 0}^n \bar{\alpha}_{n - i} \otimes \bar{\alpha}_i + \sum_{i = 0}^{n-1} u (e_\sigma \bar{\alpha}_{n - 1 - i} \otimes e_\sigma \bar{\alpha}_{i}) \\
\end{align*}
For finite sequences 
$$J = (j_\sigma, j_0, j_1, \cdots) \qquad j_k \geq 0,$$
define
$$(e_\sigma \bar{\alpha})^J = e_\sigma^{\circ j_\sigma} \circ \bar{\alpha}_{(0)}^{\circ j_0} \circ \bar{\alpha}_{(1)}^{\circ j_1} \cdots$$
where the $\circ$ product comes from the pairing $\circ: K_V \wedge K_W \to K_{V + W}.$

\begin{theorem} \label{Thm:signed} Then
$$H_\star K_{*\sigma} \cong \otimes_J E[(e_\sigma \bar{\alpha})^J]$$
As an algebra where the tensor product is over all J and the coproduct follows by Hopf ring properties from the $\bar{\alpha}$'s. 
\end{theorem}

\begin{proof}
For finite sequences 
$$J = (j_\sigma, j_0, j_1, \cdots) \qquad j_k \geq 0,$$
define $ \vert \vert J \vert \vert = \sum j_k$ (including the $\sigma$ subscript) and 
$$(e_\sigma \bar{\alpha})^J = e_\sigma^{\circ j_\sigma} \circ \bar{\alpha}_{(0)}^{\circ j_0} \circ \bar{\alpha}_{(1)}^{\circ j_1} \cdots.$$ Consider elements 
$$ (e_\sigma \bar{\alpha})^J $$
with $\vert \vert J \vert \vert = n$ in the homology of $B^\sigma K_{(n-1) \sigma}.$

To show these elements in fact form a free basis for the homology, we show that they satisfy the conditions of the Behrens-Wilson computational lemma. 
 The map to the underlying homology, $H_\star K_{n \sigma} \to H_* K_n,$ , the underlying homology of $H_\star K_{n \sigma},$ is given by
$$(e_\sigma \bar{\alpha})^J \mapsto (e_1 \alpha)^J.$$
The map on fixed points $H_\star K_{n \sigma} \to H_* K_{n\sigma}^{C_2},$ is given by
$$(e_\sigma \bar{\alpha})^J \mapsto e_0^{\circ j_\sigma} \circ a_{(0)}^{\circ j_0} \circ a_{(1)}^{\circ j_1} \cdots.$$
Thus these elements form a from a free basis for $H_\star K_{n \sigma}.$ 

We deduce the multiplicative ring structure using a Hopf ring argument due to Ravenel and Wilson \cite{Wilson1982}. Each $(e_\sigma \bar{\alpha})^J$ can be written as $e_\sigma^{\circ j_\sigma} \circ \bar{\alpha}_{(0)}^{\circ j_0} \circ \bar{\alpha}_{(1)}^{\circ j_1} \cdots \bar{\alpha}_{(n)}^{\circ j_n}$ where $n$ is some nonnegative integer or $n = \sigma.$  
By the distributive law (\ref{eq:distlaw}),
$$(e_\sigma \bar{\alpha})^J * (e_\sigma \bar{\alpha})^J = e_\sigma^{\circ j_\sigma} \circ \bar{\alpha}_{(0)}^{\circ j_0} \circ \bar{\alpha}_{(1)}^{\circ j_1} \circ \cdots \circ (\bar{\alpha}_{(n)} * \bar{\alpha}_{(n)}) = 0.$$
 The coproduct is induced by the map $K_\sigma \times \cdots \times K_\sigma \to K_{n \sigma}$ which is a map of coalgebras on $H_\star.$

\end{proof}

\begin{remark}
Note that $e_0^{\circ k} = e_0$ for $k > 0$ by Lemma \ref{lem:6.2}.
\end{remark}

\section{Bar and twisted bar spectral sequence computations}\label{sec:btwbssComps}

The first half of this section focuses on the $RO(C_2)$-graded bar spectral sequence. We describe the $d_1$-differentials, the Tor term coinciding with the $E^2$-page, and Hopf ring structure present in the spectral sequences computing $H_\star K_V$ when $\sigma + 1 \subset V.$

In the second half of this section, we study the analogous twisted spectral sequence, giving evidence of arbitrarily long equivariant degree shifting differentials appearing computations of the $RO(C_2)$-graded homology of the spaces $K_{*\sigma}.$ We describe how these differentials appear to arise in a structured way involving the norm.

\subsection{The $RO(C_2)$-graded bar spectral sequence}

The $RO(C_2)$-graded bar spectral sequences arises via a filtered complex in the same way as the ordinary integer graded version. The bar construction $B$ on a topological monoid $A,$ is filtered by 
$$ B^{[t]} A \simeq \underset{t \geq n \geq 0}{\amalg} \Delta^n \times A^n / \sim \qquad \subset B A$$
with associated graded pieces
$$\displaystyle{\nicefrac{B^{[t]} A}{B^{[t-1]}} A} \simeq S^t \wedge A^{\wedge t}.$$
Applying $H_\star(-)$ to these filtered spaces gives the $RO(C_2)$-graded bar spectral sequence with $E^1$-page
$$E^1_{t, \star} = {H}(S^t) \otimes {H}_\star(A)^{\otimes t}, $$
computing $H_\star(B A).$ This $RO(C_2)$-graded bar spectral sequence has
$$E^2_{*,\star} \simeq \text{Tor}_{*,\star}^{H_\star K_V} ({H\umF_2}_\star, {H\umF_2}_\star) \Longrightarrow H_\star B K_V \cong H_\star K_{V + 1}.$$

This spectral sequence behaves similarly to the integer graded version and collapses in many examples, including those for $BS^1 \simeq \mC P^\infty,$ $BS^\sigma \simeq \mC P^\infty_{tw}, $ and $B K_0 \simeq \mR P^\infty$  (Example \ref{ROGbssCollapsing}), which collapse for degree reasons.
\begin{example}
	$$\begin{aligned} \label{ROGbssCollapsing} 
		H_\star \mathbb{C} P^\infty & = E[\beta_{(0)}, \beta_{(1)}, \cdots ] = \Gamma [ e_2 ] \textrm{ where } \vert \beta_{(i)} \vert = 2^{i + 1},\\
		H_\star \mathbb{C} P_{tw}^\infty & = E[ \bar{\beta}_{(0)}, \bar{\beta}_{(1)}, \cdots ] = \Gamma [ e_\rho ] \textrm{ where } \vert \bar{\beta}_{(i)} \vert = \rho 2^i\nonumber \\
		H_\star \mathbb{R}P^\infty & = E [ e_1, \alpha_{(0)}, \alpha_{(1)}, \cdots ] \textrm{ where } \vert e_1 \vert = 1 \textrm{ and } \vert \alpha_{(i)} \vert = 2^{i + 1}
	\end{aligned}$$
\end{example}

\subsection{The $RO(C_2)$-graded bar spectral sequence: $d_1$-differentials}

The classical bar construction does not introduce any group action hence the $d_1$-differentials in the $RO(C_2)$-graded bar spectral sequence behave in almost the same as those in the underlying integer graded spectral sequence. The difference is that the cycles supporting $d_1$-differentials in the $RO(C_2)$-graded spectral sequence are representation degree shifted copies of the $RO(C_2)$-graded homology of the point and their targets are the same. This is in contrast with the integer graded case where the differentials are maps of non-graded rings. For example, all $d_1$ differentials in the $RO(C_2)$-graded case look behave like those shown in Figure \ref{Fig:barss}. 
\begin{figure}[ht] 
	\begin{center}
		\begin{tikzpicture}[scale=0.5]
			
			\coordinate (Origin)   at (0,0);
			\coordinate (XAxisMin) at (-2,0);
			\coordinate (XAxisMax) at (9,0);
			\coordinate (YAxisMin) at (0,-2);
			\coordinate (YAxisMax) at (0,9);
			\draw [thick, gray,-latex] (XAxisMin) -- (XAxisMax) node[anchor=west] {$p$};
			\draw [thick, gray,-latex] (YAxisMin) -- (YAxisMax) node[anchor=east] {$q$};
			
			\foreach \y in {-1,0,1,2,3,4,5,6,7,8}\draw [gray] (-0.1, \y) -- (0.1, \y) node[anchor=east] {$\y$};
			
			\foreach \x in {-1,1,2,3,4,5,6,7,8}\draw [gray] (\x, -0.1) -- (\x, 0.1) node[anchor=south] {$\x$};
			
			\draw [line width = 0.5mm, blue] (4,-2) -- (4,3);
			\draw [line width = 0.5mm, blue] (-1,-2) -- (4,3);
			\node [] (e) at (4, -2.5) {\color{blue} $[xy] $ \color{black} } ;
			\draw [line width = 0.5mm, blue] (4,5) -- (4,9);
			\draw [line width = 0.5mm, blue] (4,5) -- (8,9); 
			
			\draw [line width = 0.5mm, blue] (5,-2) -- (5,3);
			\draw [line width = 0.5mm, blue] (0,-2) -- (5,3);
			\node [] (e) at (6.25, 2.25) {\color{blue} $[x \vert y] $ \color{black} } ;
			\draw [line width = 0.5mm, blue] (5,5) -- (5,9);
			\draw [line width = 0.5mm, blue] (5,5) -- (8,8);
			
			\draw [-stealth, green, line width = 0.5 mm] (5,3) -- (4,3);
			\node [] (a) at (4.5, 3.5) {\color{green} \large $\mathbf{d_1} $ \color{black}};
			
		\end{tikzpicture}
		\caption{Example: A $d_1$-differential in the $RO(C_2)$-graded bar spectral sequence}
		\label{Fig:barss}
	\end{center}
\end{figure}
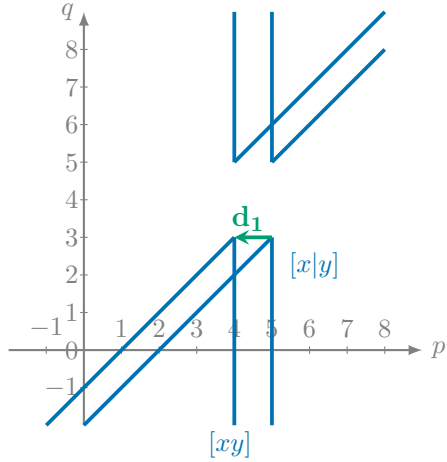

In greater specificity, Figure \ref{Fig:barssdiffmoredeets} 
shows a $d_1$ differential in the $RO(C_2)$-graded bar spectral sequence 
$$E^2_{*,\star} \simeq \text{Tor}_{*,\star}^{H_\star K_\sigma} ({H\umF_2}_\star, {H\umF_2}_\star) \Longrightarrow H_\star B K_\sigma \cong H_\star K_{\sigma + 1}$$ 
computing the $RO(C_2)$-graded homology of $K_\rho.$
In the figure, $x := e_\sigma$ with $\vert x \vert = \sigma$ and $y := \bar{\alpha}_{(0)}$ with $\vert y \vert = \rho.$ The two double cones shown are supported by the bar representatives $[xy]$ and $[x \, \vert \, y].$ The $d_1$-differential maps from the unit of the infinite dimensional graded ring ${H \umF_2}_\star$ supported by $[x \, \vert \, y]$ onto the unit of the $RO(C_2)$-graded homology of a point supported by the bar representative $[xy].$ Figure \ref{Fig:barssdiffmoredeets} depicts that this map of units in fact induces a map of graded rings surjecting onto the copy of the $RO(C_2)$-graded homology of a point supported by $[xy].$ 
\begin{figure}[ht] 
	\begin{center}
		\begin{tikzpicture}[scale=0.5]
			
			\coordinate (Origin)   at (0,0);
			\coordinate (XAxisMin) at (-2,0);
			\coordinate (XAxisMax) at (9,0);
			\coordinate (YAxisMin) at (0,-2);
			\coordinate (YAxisMax) at (0,9);
			\draw [thick, gray,-latex] (XAxisMin) -- (XAxisMax) node[anchor=west] {$p$};
			\draw [thick, gray,-latex] (YAxisMin) -- (YAxisMax) node[anchor=east] {$q$};
			
			\foreach \y in {-1,0,1,2,3,4,5,6,7,8}\draw [gray] (-0.1, \y) -- (0.1, \y) node[anchor=east] {$\y$};
			
			\foreach \x in {-1,1,2,3,4,5,6,7,8}\draw [gray] (\x, -0.1) -- (\x, 0.1) node[anchor=south] {$\x$};
			
			\draw [line width = 0.5mm, blue] (4,-2) -- (4,3);
			\draw [line width = 0.5mm, blue] (-1,-2) -- (4,3);
			\node [] (e) at (4, -2.5) {\color{blue} $[xy] $ \color{black} } ;
			\draw [line width = 0.5mm, blue] (4,5) -- (4,9);
			\draw [line width = 0.5mm, blue] (4,5) -- (8,9); 
			
			\draw [line width = 0.5mm, blue] (5,-2) -- (5,3);
			\draw [line width = 0.5mm, blue] (0,-2) -- (5,3);
			\node [] (e) at (6.5, 2.5) {\color{blue} $[x \vert y] $ \color{black} } ;
			\draw [line width = 0.5mm, blue] (5,5) -- (5,9);
			\draw [line width = 0.5mm, blue] (5,5) -- (8,8);
			
			\draw [-stealth, green, line width = 0.5 mm] (5,3) -- (4,3);
			\node [] (a) at (4.5, 3.5) {\color{green} \large $\mathbf{d_1} $ \color{black}};
			
			\draw [-stealth, green, line width = 0.5 mm] (5,5) -- (4,5);
			\draw [-stealth, green, line width = 0.5 mm] (5,6) -- (4,6);
			\draw [-stealth, green, line width = 0.5 mm] (5,7) -- (4,7);
			\draw [-stealth, green, line width = 0.5 mm] (5,8) -- (4,8);
			
			\draw [-stealth, green, line width = 0.5 mm] (6,6) -- (5,6);
			\draw [-stealth, green, line width = 0.5 mm] (7,7) -- (6,7);
			\draw [-stealth, green, line width = 0.5 mm] (8,8) -- (7,8);
			
			\draw [-stealth, green, line width = 0.5 mm] (4,2) -- (3,2);
			\draw [-stealth, green, line width = 0.5 mm] (3,1) -- (2,1);
			\draw [-stealth, green, line width = 0.5 mm] (2,0) -- (1,0);
			\draw [-stealth, green, line width = 0.5 mm] (1,-1) -- (0,-1);
			
			\draw [-stealth, green, line width = 0.5 mm] (5,2) -- (4,2);
			\draw [-stealth, green, line width = 0.5 mm] (5,1) -- (4,1);
			\draw [-stealth, green, line width = 0.5 mm] (5,0) -- (4,0);
			\draw [-stealth, green, line width = 0.5 mm] (5,-1) -- (4,-1);
			
		\end{tikzpicture}
		\caption{A more detailed picture of a $d_1$-differential in the $RO(C_2)$-graded bar spectral sequence}
		\label{Fig:barssdiffmoredeets}
	\end{center}
\end{figure}
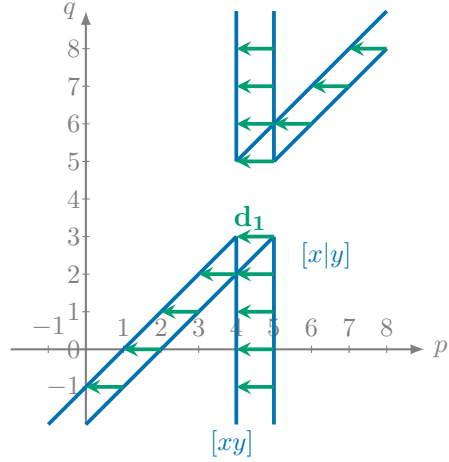

\subsection{Hopf ring structure in the $\bm{RO(C_2)}$-graded bar spectral sequence and $\mathbf{H_\star K_V,}$ $\mathbf{\bm{\sigma + 1 \subset} V}$}

In Theorem \ref{Thm:signed}, we computed $H_\star K_{n \sigma}$, showing that it is free over $H_\star$. To compute $H_\star K_V$ for real representations $V \cong i + j \sigma,$ we consider $\circ$-product structure in the $RO(C_2)$-graded bar spectral sequence
$$E_{*,*}^2 \simeq \text{Tor}_{*, *}^{H_\star K_V} (H_\star, H_\star) \Rightarrow H_\star K_{V + 1}, $$	
and observe that theorems of Thomason and Wilson extend directly from the nonequivariant integer graded setting to the $C_2$-equivariant $RO(C_2)$-graded setting. In Theorem \ref{thm:diffpair}, we need an additional flatness hypothesis to account for $H_\star(X; \umF_2)$ not necessarily being flat, unlike $H_*(X; \mF_2).$

\begin{theorem}[\cite{ThomasonWilson1980}]
	The $\circ$ product factors as 
	\begin{align*}
	B_t K_V \times K_W \to B_t K_{V + W} \\
	\bigcap \qquad \qquad \qquad \bigcap \qquad \\
	\circ: B K_V \times K_W \to B K_{V + W}
	\end{align*}
	and the map 
	\begin{align*}
	\displaystyle{\nicefrac{B_t K_V}{B_{t-1} K_V}} \times K_W \to \displaystyle{\nicefrac{B_t K_{V + W}}{B_{t-1} K_{V + W}}}  \\
	\vertsimeq \qquad \qquad \qquad \vertsimeq \qquad \qquad \qquad \\
	S^{t} \wedge K_V^{\wedge t} \times K_W \to S^{t} \wedge K_{V+W}^{\wedge t} \qquad
	\end{align*}
	is described inductively as $(k_1, \cdots, k_t) \circ k = (k_1 \circ k, \cdots, k_t \circ k).$
\end{theorem}

\begin{theorem}[\cite{ThomasonWilson1980}] \label{thm:diffpair}
	Let $E_{*, \star}^r (E_\star K_V) \implies E_\star K_{V + 1}$ be the bar spectral sequence and suppose $E^r$ is $H_\star$-flat for all $i \leq r$. Compatible with 
	$$\circ: E_\star K_{V + 1} \otimes_{H_\star} E_\star K_W \to E_\star K_{V + W + 1}$$
	is a pairing
	\begin{align} \label{barsspairing}
	E_{t,\star}^r (E_\star K_V) \otimes_{H_\star} E_\star K_W \to E_{t, \star}^r (E_\star K_{V + W}) 
	\end{align}
	with $d^r(x) \circ y = d^r (x \circ y).$ For $r=1$ this pairing is given by
	$$(k_1 \vert \cdots \vert k_t) \circ k = \sum \pm (k_1 \circ k' \vert k_2 \circ k'' \vert \cdots \vert k_s \circ k^{(t)})$$
	where $k \to \sum k' \otimes k'' \otimes \cdots \otimes k^{(t)}$ is the iterated reduced coproduct. 
\end{theorem}

\begin{theorem} \label{Thm:HKVgen}
	The $RO(C_2)$-graded homology of $K_V,$ where $\sigma + 1 \subset V,$ is exterior on generators given by the cycles on the $E^2$-page of the $RO(C_2)$-graded spectral sequence.  
\end{theorem}

\begin{proof}
	Let $E_{*, \star}^r (E_\star K_V) \implies E_\star K_{V + 1}$ be the bar spectral sequence and $\Delta:K_V \to K_V \times K_V$ be the diagonal map. If $E^r$ is $H_\star$-flat for all $i \leq r,$ then there is a natural transformation $\mu: E^r(X) \otimes E^r(Y) \to E^r(X \times Y)$ and the coalgebra structure on $E^r$ is given by $\mu^{-1} \Delta_*.$
	
	Suppose $E^r_{*,\star}$ where $r>2$ is the first page after the $E^2_{*, \star}$-page with a nonzero differential. Then $E^r_{*,\star} = E^2_{*, \star} \cong \textrm{Tor}_{*,*}^{H_\star K_V} (H_\star, H_\star)$ which is a coalgebra, so $\mu$ is an isomorphism and the differentials $d_r$ satisfy the Leibniz and coLeibniz rules.
	
	Consider the shortest nonzero differential $d_r$ in lowest topological degree. If such a differential exists, it must map from an algebra indecomposable to a coalgebra primitive. To see this, we recount a classical Hopf ring argument, which also appears in \cite{RavanelWilson1980} and \cite{AngVigRog2005}. Suppose $d_r(xy) \neq 0$ and $xy$ is in lowest topological degree. Then 
	$$d_r(xy) = d^r(x)y \pm x d_r(y) $$
	so $d_r (x)$ or $d_r(y)$ are nonzero contradicting that $xy$ is in lowest topological degree. Dually, if $d_r(z)$ is not a coalgebra primitive, then 
	$$\psi (z) = z | 1 + 1 | z + \Sigma z_i' | z_i" $$
	and the coLeibniz formula 
	$$\psi \circ d_r = (d_r | 1 \pm 1 | d_r)\psi $$
	implies $d_r(z_i')$ or $d_r(z_i")$ is nonzero, contradicting that $z$ is in lowest topological degree. 
	
	There are no coalgebra primitives on $E^2_{*, \star} = E^r_{*,\star} $ due to the coproduct structure on $H_\star K_\sigma.$ Thus there are no nontrivial differentials and the spectral sequence collapses. 
	
	Let $x$ be a cycle on $E^2_{*, \star}.$ To show there are no extension problems, we only need to show $$x * x = 0.$$ The multiplication by $2$ map $2: K_V \to K_V,$ which factors as the composition
	$$K_V  \xrightarrow[]{\quad \Delta \quad} K_V \times K_V \xrightarrow{*} K_V$$
	is homotopically trivial so
	$$0 = 2_\star: H_\star K_V \to H_\star K_V.$$
	Consider the coproduct structure on $H_\star K_{* \sigma}$ and $E^2_{*, \star}.$ There is a cycle $y$ on $E^2_{*, \star},$ with the symmetric term of the coproduct $\psi(y)$ equal to $x \otimes x.$ This means there is $y$ such that $2_\star y = x * x$ so $x * x= 0$ as desired. 
	
\end{proof}

\subsection{Circle product names for the generators of $\bm{H_\star K_{\sigma + i}}$} \label{sc:KsigmaI}
We give names to the generators of $H_\star K_{\sigma + i}$ and indicate how the bookkeeping becomes increasingly complicated as the number of sign representations in $V$ where $1 + \sigma \subset V$ increases (Example \ref{ex:moreGens}).  

To write these answers, we recall our notation for $H_\star K_\rho$. Let 
$$\bar{\beta}_{i} \in H_{\rho i} K(\umZ, \rho), \, \qquad (i \geq 0).$$
This gives additional generators,
$$\bar{\beta}_{(i)} = \bar{\beta}_{2^i} \quad (i \geq 0), $$
of $H_\star K_\rho$ with coproduct  
$$ \psi( \bar{\beta}_{n} ) = \sum_{i = 0}^n \bar{\beta}_{n - i} \otimes \bar{\beta}_i.$$
Then for finite sequences 
$$ I = (i_1, i_2, \cdots, i_k ), \qquad 0 \leq i_1 < i_2 < \cdots, $$
$$ W = (w_1, w_2, \cdots, w_q), \qquad 0 \leq w_1 < w_2 < \cdots,$$ 
$$	J = (j_{-1}, j_0, j_1, \cdots, j_\ell ), \qquad \text{where } j_{-1} \in \{0,1\} \text{ and all other} j_n \geq 0, $$
and 
$$ Y = (y_{-1}, y_0, y_1, \cdots, y_r ), \qquad \text{where } y_{-1} \in \{0,1\} \text{ and all other } y_n \geq 0, $$
define
$$(e_1 \alpha \beta)^{I,J} = e_1^{\circ j_{-1}} \circ \alpha_{(i_1)} \circ \alpha_{(i_2)} \circ \cdots \circ \alpha_{(i_k)} \circ \beta_{(0)}^{\circ j_0} \circ \beta_{(1)}^{\circ j_1} \circ \cdots \circ \beta_{(\ell)}^{\circ j_\ell},$$
$$(e_1 \alpha \beta)^{W,Y} = e_1^{\circ y_{-1}} \circ \alpha_{(w_1)} \circ \alpha_{(w_2)} \circ \cdots \circ \alpha_{(w_q)} \circ \beta_{(0)}^{\circ y_0} \circ \beta_{(1)}^{\circ y_1} \circ \cdots \circ \beta_{(r)}^{\circ j_r},$$
$$ \vert I \vert = k, \qquad \vert W \vert = q \qquad \vert \vert J \vert \vert = \Sigma j_n, \qquad \text {and } \vert \vert Y \vert \vert = \Sigma y_n.$$
Then

\begin{theorem} \label{Thm:sigmaPlusl} We have
	$$ H_\star K_{\sigma + i} \cong E[(e_1 \alpha \beta)^{I,J} \circ \bar{\alpha}_{(m)}, (e_1 \alpha \beta)^{W, Y} \circ \bar{\beta}_{(t)}] $$
	where $m > i_k$ and $m \geq l,$ $t > w_q$ and $t \geq y_r,$ $\vert I \vert + 2 \vert \vert J \vert \vert = i$ and $\vert W \vert + 2 \vert \vert Y \vert \vert = i - 1,$  and the coproduct follows by Hopf ring properties from the $\alpha_{(i)}$'s, $\beta_{(i)}$'s, $\bar{\alpha}_{(i)}$'s and $\bar{\beta}_{(i)}$'s. 
\end{theorem}

We offer a proof distinct from that of Theorem \ref{Thm:HKVgen}. 

\begin{proof}
	Apply the Behrens-Wilson lemma to the generators $ (e_1 \alpha \beta)^{I,J} \circ \bar{\alpha}_{(m)}$ and $(e_1 \alpha \beta)^{W, Y} \circ \bar{\beta}_{(t)}$ defined in the theorem. The map to the underlying homology is clear as the generators have no $a$-torsion. On fixed points, 
	$$ (e_1 \alpha \beta)^{I,J} \circ \bar{\alpha}_{(m)} \mapsto (e_1 \alpha \beta)^{I,J} \circ a_{(m)} $$
	and 
	$$ (e_1 \alpha \beta)^{W, Y} \circ \bar{\beta}_{(t)} \mapsto (e_1 \alpha \beta)^{W, Y} \circ a_{(t)}, $$
	giving a basis for $K_{\sigma + i}^{C_2} \simeq K_{i + 1} \times K_i$ where the $a_{(i)}$ are notation for the underlying nonequivariant homology of $K_\sigma$ (see \textsection \ref{subsec:underlyinghom}). The multiplicative and comultiplicative structures are deduced similarly to Theorem \ref{Thm:signed}.
\end{proof}

\begin{example} \label{ex:moreGens}
\normalfont 
Consider the $RO(C_2)$-graded bar spectral sequence
$$ E^2_{*,\star} \simeq \text{Tor}_{*,\star}^{H_\star K_{2 \sigma}} ({H\umF_2}_\star, {H\umF_2}_\star) \Longrightarrow H_\star B K_{2 \sigma} \cong H_\star K_{2 \sigma + 1}.$$ 

The indecomposable cycles on the $E^2$-page are
$$ [ e_\sigma \circ e_\sigma ], \qquad [ e_\sigma \circ \bar{\alpha}_{(i)} ], \qquad [ \bar{\alpha}_{(j_1)} \circ \bar{\alpha}_{(j_2)} ]$$
$$\phi^{(k)} (e_\sigma \circ e_\sigma), \qquad \phi^{(k)} (e_\sigma \circ \bar{\alpha}_{(i)}), \qquad \phi^{(k)} (\bar{\alpha}_{(j_1)} \circ \bar{\alpha}_{(j_2)}),$$
where $\phi^{(k)}(x)$ is notation for the bar representative $\underbrace{ [ x \vert \cdots \vert x]}_{2^i\text{ - copies}}.$ 

Since trivial representation suspension is $\circ$ multiplication with $e_1,$ we can identify $ [ e_\sigma \circ e_\sigma ]$ with $e_1 \circ e_\sigma \circ e_\sigma,$ $ [ e_\sigma \circ \bar{\alpha}_{(i)} ]$ with $e_1 \circ  e_\sigma \circ \bar{\alpha}_{(i)},$ and $[ \bar{\alpha}_{(j_1)} \circ \bar{\alpha}_{(j_2)} ]$ with $e_1 \circ \bar{\alpha}_{(j_1)} \circ \bar{\alpha}_{(j_2)}.$ Using the bar spectral sequence pairing (\ref{barsspairing}) for 
$$B K_\sigma \times K_{\sigma} \to BK_{2\sigma},  $$
we can identify $ \phi^{(k)} (e_\sigma \circ \bar{\alpha}_{(i)})$ with $\bar{\beta}_{(j_1 + 1)} \circ \bar{\alpha}_{(j_2 + 1)},$
and using the bar spectral sequence pairing (\ref{barsspairing}) for
$$B K_0 \times K_{2 \sigma} \to B K_{2\sigma} $$
we can identify $\phi^{(k)} (\bar{\alpha}_{(j_1)} \circ \bar{\alpha}_{(j_2)})$ with $ \alpha_{(i_1)} \circ \bar{\alpha}_{(i_2)} \circ \bar{\alpha}_{(j_3)}.$
By Theorem \ref{Thm:HKVgen}, the $\phi^{(k)} (e_\sigma \circ e_\sigma)$ are also permanent cycles. However, degree reasons make it impossible to identify them in terms of circle products (there are too many sign representations) and thus we have a new family of generators which are not circle products of elements in $K_1, K_\sigma, \, K_\rho, \,$ or $K_{2\sigma}.$ 

\begin{cor}
	We have 
	\begin{align*}
	H_\star K_{2\sigma + 1} \cong E[ & e_1 \circ e_\sigma \circ e_\sigma , \, e_1 \circ  e_\sigma \circ \bar{\alpha}_{(i)}, \, e_1 \circ \bar{\alpha}_{(j_1)} \circ \bar{\alpha}_{(j_2)}, \\
	& \bar{\beta}_{(j_1 + 1)} \circ \bar{\alpha}_{(j_2 + 1)}, \, \alpha_{(i_1)} \circ \bar{\alpha}_{(i_2)} \circ \bar{\alpha}_{(j_3)}, \, \phi^{(k)} (e_\sigma \circ e_\sigma)] 
	\end{align*}
	as an algebra where the coproduct follows by Hopf ring properties from the $\alpha_{(i)}$'s, $\beta_{(i)}$'s, $\bar{\alpha}_{(i)}$'s, $\bar{\beta}_{(i)}$'s and coproduct structure on $\textrm{Tor}_{*, \star}^{H K_{2\sigma}} ({H \umF_2}_\star, {H \umF_2}_\star).$
\end{cor}

\end{example}

As the number of sign representations in $V$ where $1 + \sigma \subset V$ increases, the number of additional generators grows, making bookkeeping and identifying homology generators in terms of the bar spectral sequence pairing \ref{barsspairing} an increasingly complicated task. 

\subsection{The $RO(C_2)$-graded twisted bar spectral sequence}

We now turn to the twisted analogue of the $RO(C_2)$-graded bar spectral sequence. Similar to the classical case, the twisted bar construction $B^\sigma A$ is filtered by 
$$ (B^\sigma A)^{[t]} \simeq \coprod_{t \geq n \geq 0} \Delta^n \times A^n / \sim \qquad \subset B^\sigma A$$
with associated graded pieces
$$\displaystyle{\nicefrac{(B^\sigma A)^{[t]}}{(B^\sigma A)^{[t-1]}} \simeq S^{\ceil*{\frac{t}{2}} \sigma + \floor*{\frac{t}{2}}} \wedge A^{\wedge t}}, $$
where the $C_2$-action on $A^t$ is given by $\gamma(a_1 \wedge \cdots \wedge a_n) = (\gamma a_n \wedge \cdots \wedge \gamma a_1).$ Applying $H_\star(-)$ to these filtered spaces gives the twisted bar spectral sequence
$$E_{t, \star}^1 = \tilde{H}_\star(S^{\ceil*{\frac{t}{2}} \sigma + \floor*{\frac{t}{2}}} \wedge A^t ) \Rightarrow H_\star B^\sigma A, $$
with differentials
$$d_r: E_{t, \star}^r \to E_{t - r, \star - 1}^r,$$ 
computing $H_\star(B^\sigma A).$ 

In general, this spectral sequence lacks an explicit $E^2$-page and can be difficult to compute. We give some readily computable examples which collapse on the $E^1$-page and then turn to analyzing the structure of thetwisted bar spectral sequence in examples computing the $RO(C_2)$-graded homology of $C_2$-equivariant Eilenberg-MacLane spaces. 

\begin{example}\label{twistbssExs} The $RO(C_2)$-graded twisted bar spectral sequences computing the homology of $B^\sigma \mF_2 \simeq K(\umF_2, \sigma) \simeq \mathbb{R} P_{tw}^\infty $  and $ B^\sigma S^1 \simeq K(\umZ, \rho) \simeq  \mathbb{C} P_{tw}^\infty$ collapse on the $E^1$-page. As rings,
	\begin{align*} 
		H_\star \mathbb{R} P_{tw}^\infty & = E[e_\sigma, \bar{\alpha}_{(0)}, \bar{\alpha}_{(1)}, \cdots ] = E[e_\sigma] \otimes \Gamma [ \bar{\alpha}_{(0)} ], \, \vert e_\sigma \vert = \sigma, \, \vert \bar{\alpha}_{(i)} \vert = \rho 2^i, \\
		H_\star \mathbb{C} P_{tw}^\infty & = E[ \bar{\beta}_{(0)}, \bar{\beta}_{(1)}, \cdots ] = \Gamma [ e_\rho ] \textrm{ where } \vert \bar{\beta}_{(i)} \vert = \rho 2^i. 
	\end{align*}
	We write the proof for $H_\star \mR P_{tw}^\infty$ as the computation for $H_\star \mC P_{tw}^\infty$ is similar. 
\end{example}

\begin{proof}
	We first prove the additive statement that $H_\star \mathbb{R}P^\infty_{tw}$ is a free $H_\star$-module with a single generator in each degree $\ceil{\frac{n}{2}} \sigma +  \floor{\frac{n}{2}}.$ We then show  $H_\star \mR P^\infty $ has ring structure $E[e_\sigma, \bar{\alpha}_{(0)}, \bar{\alpha}_{(1)}, \cdots ] = E[e_\sigma] \otimes \Gamma [ \bar{\alpha}_{(0)} ]$ where $\vert e_\sigma \vert = \sigma$ and $\vert \bar{\alpha}_{(i)} \vert = 2^i \rho. $ We start with the twisted bar spectral sequence 
	$$ E_{t, \star}^1 = \tilde{H}_\star(S^{\ceil*{\frac{t}{2}} \sigma + \floor*{\frac{t}{2}}} \wedge \mF_2^t) \Rightarrow H_\star B^\sigma \mF_2.$$
	Specifically, 
		\begin{align*}
			E^1_{t, \star} & \cong \tilde{H}_\star (\displaystyle{\nicefrac{B_t^\sigma \mathbb{F}_2}{B_{t-1}^\sigma}} \mathbb{F}_2 ) \\
			& \cong \tilde{H}_\star( S^{\ceil*{\frac{t}{2}} \sigma + \floor*{\frac{t}{2}}} \wedge \mathbb{F}_2^{\wedge t} ) \qquad \quad \qquad \qquad  \qquad \qquad \qquad  \qquad \quad \, \textrm{(by definition)} \\
			& \cong \tilde{H}_\star (S^{\ceil*{\frac{t}{2}} \sigma + \floor*{\frac{t}{2}}}) \otimes \tilde{H}_\star (\mathcal{N}_e^{C_2} (\mathbb{F}_2^{\wedge \floor*{\frac{t}{2}} } \wedge \mF_2^\epsilon )) \quad \textrm{(freeness \& properties of }\mathcal{N}_e^{C_2}) \\
			& \cong \tilde{H}_\star (S^{\ceil*{\frac{t}{2}} \sigma + \floor*{\frac{t}{2}}}) \otimes \tilde{H}_\star  (\mathbb{F}_2)^{\wedge t} \quad \textrm{(homology of norm of underlying free space)}
		\end{align*}
	where in the last step, since the homology of $\mF_2$ splits as the homology of induced representation spheres, the homology of the norm is the norm of the homology of the underlying space \cite{Hill2019}.
	
	Because the filtration degree $t$ corresponds to the topological degree $p$ and differentials $d^r$ shift topological degree down by one, there are no nonzero $d^r,$ $r > 1.$ There can be no nonzero $d^1$ because if there were, on passing to the nonequivariant homology of the underlying space, $H_* \mathbb{R} P^\infty,$ we would be killing a known generator which is a contradiction. Hence the homology is free with a single generator in each degree $\ceil{\frac{n}{2}} \sigma +  \floor{\frac{n}{2}}.$This $E^1$-page is depicted in Figure \ref{Fig:E1pageKsigma}.
	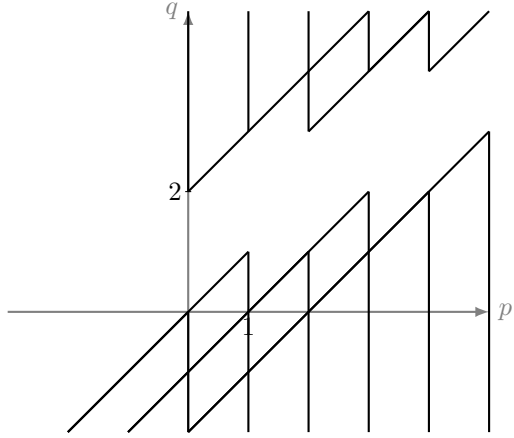
\begin{figure}[ht]	
		\begin{center}
			\begin{tikzpicture}[scale=0.8]
				
				\coordinate (Origin)   at (0,0);
				\coordinate (XAxisMin) at (-3,0);
				\coordinate (XAxisMax) at (5,0);
				\coordinate (YAxisMin) at (0,-2);
				\coordinate (YAxisMax) at (0,5);
				\draw [thick, gray,-latex] (XAxisMin) -- (XAxisMax) node[anchor=west] {$p$};
				\draw [thick, gray,-latex] (YAxisMin) -- (YAxisMax) node[anchor=east] {$q$};
				
				\draw (-0.05, 2) -- (0.05, 2) node[anchor=east, distance=0.5] {$2$};
				\draw (1, -0.05) -- (1, 0.05) node[anchor=north, distance=0.5] {$1$};
				
				\draw [thick] (0,2) -- (0,5);
				\draw [thick] (0,2) -- (3,5);
				\draw [thick] (0,0) -- (0,-2);
				\draw [thick] (0,0) -- (-2,-2);
				
				\draw [thick] (1,3) -- (1,5);
				\draw [thick] (1,1) -- (1,-2);
				\draw [thick] (1,1) -- (-2,-2);
				
				\draw [thick] (2,3) -- (2,5);
				\draw [thick] (2,3) -- (4,5);
				\draw [thick] (2,1) -- (-1,-2);
				\draw [thick] (2,1) -- (2,-2);
				
				\draw [thick] (3,4) -- (3,5);
				\draw [thick] (3,4) -- (4,5);
				\draw [thick] (3,2) -- (3,-2);
				\draw [thick] (3,2) -- (-1,-2);
				
				\draw [thick] (4,4) -- (4,5);
				\draw [thick] (4,4) -- (5,5);
				\draw [thick] (4,2) -- (4,-2);
				\draw [thick] (4,2) -- (0,-2);
				
				\draw [thick] (5,3) -- (5,-2);
				\draw [thick] (5,3) -- (0,-2);
				
			\end{tikzpicture}
			\caption{$E^1$-page of the twisted bar spectral sequence computing $H_\star K_\sigma$}
			\label{Fig:E1pageKsigma}
		\end{center}
	\end{figure}

	We deduce the multiplicative structure. There is no element in degree $2 \sigma$ so $e_\sigma$ must be exterior. The remaining exterior structure can also be deduced without appealing to Hopf rings. The multiplication by $2$ map $2: K_\sigma \to K_\sigma,$ which factors as the composition
	$$ K_\sigma \xrightarrow[]{\quad \Delta \quad} K_\sigma \times K_\sigma \xrightarrow[]{\quad * \quad} K_\sigma, $$
	is homotopically trivial, so
	$$0 = 2_\star: H_\star K_\sigma \to H_\star K_\sigma.$$
	Since $2_\star(\bar{\alpha}_{(i + 1)}) = \bar{\alpha}_{(i)} * \bar{\alpha}_{(i)},$ this proves the exterior multiplication. 
\end{proof}

\begin{theorem} \label{Thm:2sigma} The $RO(C_2)$-graded twisted bar spectral sequence computing the homology of $B^\sigma S^\sigma \simeq K(\umZ, 2 \sigma)$ collapses on the $E^1$-page. As a ring,
	\begin{align*}
		H_\star K(\underline{\mathbb{Z}}, 2 \sigma) & = E[ e_{2\sigma} ] \otimes \Gamma[ \bar{x}_{(0)}] \textrm{ where } \vert e_{2\sigma} \vert = 2\sigma, \, \vert \bar{x}_{(0)} \vert = 2 \rho.
	\end{align*} 
\end{theorem}
\noindent The proof of Theorem \ref{Thm:2sigma} is analogous to the computation of $H_\star \mR P_{tw}^\infty$ given in Example \ref{twistbssExs}. 

\subsection{Higher differentials in the $RO(C_2)$-graded twisted bar spectral sequence}\label{Subsec:highertwdiffs}

In this section, we use our understanding of $H_\star K_{* \sigma}$ to analyze the structure of the twisted bar spectral sequence and find evidence of arbitrarily long equivariant degree shifting differentials.

Consider the $RO(C_2)$-graded twisted bar spectral sequence
$$ E_{t, \star}^1 = \tilde{H}_\star(S^{\ceil*{\frac{t}{2}} \sigma + \floor*{\frac{t}{2}}} \wedge K_\sigma^{\wedge t} ) \Rightarrow H_\star B^\sigma K_\sigma,$$
computing $H_\star K_{2 \sigma}.$ There are two basic building blocks in this spectral sequence. Twisted bar representatives $[ z_1 \, | \, \cdots \, | \, z_n ],$ $z \in H_\star K_\sigma$ that are fixed under the $C_2$-action of the twisted bar construction and those that possess non-trivial $C_2$-action. The twisted bar representatives which are fixed support a full double cone, that is an $RO(C_2)$-graded representation degree shifted copy of the homology of the point. An example where $| z_1 | = \sigma$ and $| z_2 | = \rho$ is shown in Figure \ref{Fig:twbssdoublec}.
\begin{figure}
	\begin{center}
		\begin{tikzpicture}[scale=0.5]
			
			\coordinate (Origin)   at (0,0);
			\coordinate (XAxisMin) at (-4,0);
			\coordinate (XAxisMax) at (12,0);
			\coordinate (YAxisMin) at (0,-2);
			\coordinate (YAxisMax) at (0,9);
			\draw [thick, gray,-latex] (XAxisMin) -- (XAxisMax) node[anchor=west] {$p$};
			\draw [thick, gray,-latex] (YAxisMin) -- (YAxisMax) node[anchor=east] {$q$};
			
			\foreach \y in {-1,0,1,2,3,4,5,6,7,8}\draw [gray] (-0.1, \y) -- (0.1, \y) node[anchor=east] {$\y$};
			
			\foreach \x in {--3,-2,-1,1,2,3,4,5,6,7,8,9,10,11,12}\draw [gray] (\x, -0.1) -- (\x, 0.1) node[anchor=north] {$\x$};
			
			\draw [line width = 0.5mm, blue] (4,-2) -- (4,3);
			\draw [line width = 0.5mm, blue] (-1,-2) -- (4,3);
			\node [] (e) at (4, -2.5) {\color{blue} $[z_1] $ \color{black} } ;
			\draw [line width = 0.5mm, blue] (4,5) -- (4,9);
			\draw [line width = 0.5mm, blue] (4,5) -- (8,9);
			
			\draw [line width = 0.5mm, blue] (8,-2) -- (8,4);
			\draw [line width = 0.5mm, blue] (2,-2) -- (8,4);
			\node [] (e) at (8, -2.5) {\color{blue} $[z_2 \vert z_2]$ \color{black} } ;
			\draw [line width = 0.5mm, blue] (8,6) -- (8,9);
			\draw [line width = 0.5mm, blue] (8,6) -- (11,9);
			
			\draw [line width = 0.5mm, blue] (10,-2) -- (10,5);
			\draw [line width = 0.5mm, blue] (3,-2) -- (10,5);
			\node [] (e) at (12.5, 4.5) {\color{blue} $[z_1 \vert z_2 \vert z_2 \vert z_1]$ \color{black} } ;
			\draw [line width = 0.5mm, blue] (10,7) -- (10,9);
			\draw [line width = 0.5mm, blue] (10,7) -- (12,9); 
		\end{tikzpicture}
		\caption{Twisted bar representatives fixed under the $C_2$-action support full double cones}
		\label{Fig:twbssdoublec}
	\end{center}
\end{figure}
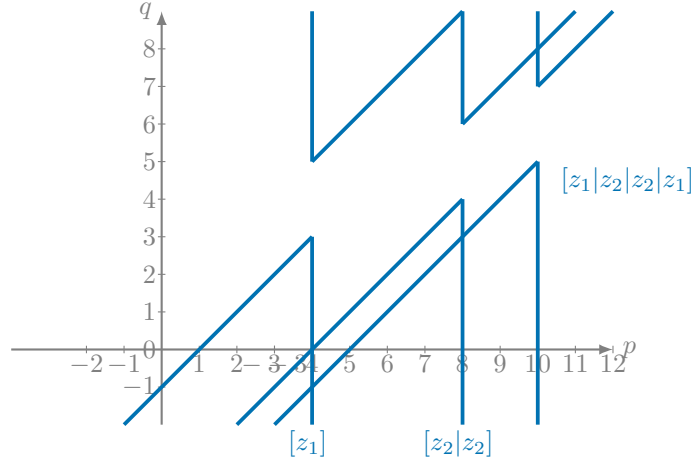
Let $\gamma$ denote the generator of $C_2.$ The remaining twisted bar representatives come in pairs $[ z_1 \, | \, \cdots \, | \, z_n ]$ and $\gamma \cdot [ z_1 \, | \, \cdots \, | \, z_n ].$ Each pair gives a copy of ${C_2}_+$ and we choose a single twisted bar representative to represent each copy. In the twisted bar spectral sequence, the representatives $[ z_1 \, | \, \cdots \, | \, z_n ]$ with non-trivial $C_2$-action support shifted degree copies of $H_\star {C_2}_+$ as depicted in Figure \ref{Fig:yellow}.
\begin{figure}[ht] 
	\begin{center}
		\begin{tikzpicture}[scale=0.5]
			
			\coordinate (Origin)   at (0,0);
			\coordinate (XAxisMin) at (-4,0);
			\coordinate (XAxisMax) at (12,0);
			\coordinate (YAxisMin) at (0,-2);
			\coordinate (YAxisMax) at (0,9);
			\draw [thick, gray,-latex] (XAxisMin) -- (XAxisMax) node[anchor=west] {$p$};
			\draw [thick, gray,-latex] (YAxisMin) -- (YAxisMax) node[anchor=east] {$q$};
			
			\foreach \y in {-1,0,1,2,3,4,5,6,7,8}\draw [gray] (-0.1, \y) -- (0.1, \y) node[anchor=east] {$\y$};
			
			\foreach \x in {--3,-2,-1,1,2,3,4,5,6,7,8,9,10,11,12}\draw [gray] (\x, -0.1) -- (\x, 0.1) node[anchor=north] {$\x$};
			
			\draw [line width = 0.5mm, yellow] (5,-2) -- (5,9);
			\node [] (f) at (5, -3.5) {\color{yellow} \large $[z_1 \vert z_2]$ \color{black}}; 
			
			\draw [line width = 0.5mm, yellow] (8.9,-2) -- (8.9,9);
			\draw [line width = 0.5mm, yellow] (9.1,-2) -- (9.1,9);
			\node [] (g) at (9, -3.5) {\color{yellow} \large $[z_1 \vert z_2 \vert z_1 z_2]$ \color{black}}; 
			\node [] (g) at (9, -4.5) {\color{yellow} \large $[z_2 \vert z_1 \vert z_1 z_2]$ \color{black}}; 
			
		\end{tikzpicture}
		\caption{Twisted bar representatives with nontrivial $C_2$ action support copies of $H_\star {C_2}_+.$}
		\label{Fig:yellow}
	\end{center} 
\end{figure}
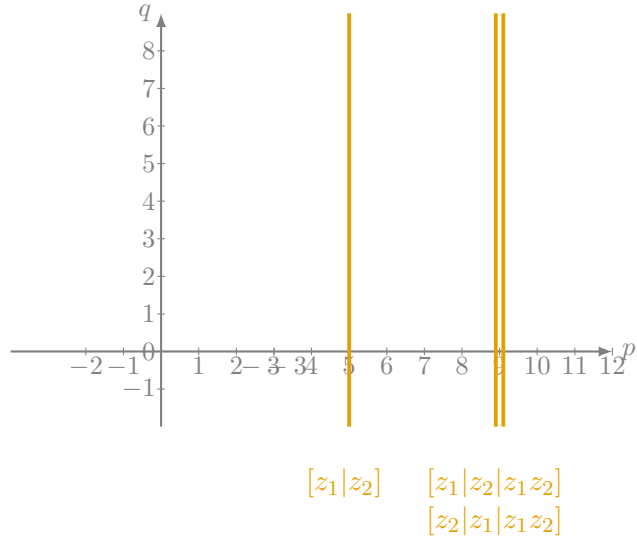

A portion of the twisted bar spectral sequence computing $H_\star K_{2\sigma}$ appears in Figure \ref{fig:2sigmaDiffs}, where $x$ represents $e_\sigma$ and $y$ represents $\bar{\alpha}_{(0)}.$
\begin{figure}[ht] 
	\begin{center}
		\begin{tikzpicture}[scale=0.5]
			
			\coordinate (Origin)   at (0,0);
			\coordinate (XAxisMin) at (-4,0);
			\coordinate (XAxisMax) at (12,0);
			\coordinate (YAxisMin) at (0,-2);
			\coordinate (YAxisMax) at (0,9);
			\draw [thick, gray,-latex] (XAxisMin) -- (XAxisMax) node[anchor=west] {$p$};
			\draw [thick, gray,-latex] (YAxisMin) -- (YAxisMax) node[anchor=east] {$q$};
			
			\foreach \y in {-1,0,1,2,3,4,5,6,7,8}\draw [gray] (-0.1, \y) -- (0.1, \y) node[anchor=east] {$\y$};
			
			\foreach \x in {--3,-2,-1,1,2,3,4,5,6,7,8,9,10,11,12}\draw [gray] (\x, -0.1) -- (\x, 0.1) node[anchor=north] {$\x$};
			
			\draw [line width = 0.5mm, blue] (4,-2) -- (4,3);
			\draw [line width = 0.5mm, blue] (-1,-2) -- (4,3);
			\node [] (e) at (4, -2.5) {\color{blue} $[xy] $ \color{black} } ;
			\draw [line width = 0.5mm, blue] (4,5) -- (4,9);
			\draw [line width = 0.5mm, blue] (4,5) -- (8,9);
			
			\draw [line width = 0.5 mm, green, dashed] (4,3) -- (5,4);
			\node[circle,draw, green, line width=0.5 mm, fill=white, outer sep=0pt, inner sep=1.5pt] (c) at (5,4){}; 
			\node [] (b) at (3.25,4) { \color{green} $[x] * [y]$ \color{black}};
			\draw [-stealth, green, line width = 0.5 mm] (5,5) -- (4,5);
			\draw [-stealth, green, line width = 0.5 mm] (5,6) -- (4,6);
			\draw [-stealth, green, line width = 0.5 mm] (5,7) -- (4,7);
			\draw [-stealth, green, line width = 0.5 mm] (5,8) -- (4,8);
			\node [] (a) at (4.5, 8.5) {\color{green} \large $\mathbf{d_1} $ \color{black}}; 
			
			\draw [line width = 0.5mm, yellow] (5,-2) -- (5,9);
			\node [] (f) at (5, -3.5) {\color{yellow} \large $[x \vert y]$ \color{black}}; 
			
			\draw [line width = 0.5mm, blue] (8,-2) -- (8,4);
			\draw [line width = 0.5mm, blue] (2,-2) -- (8,4);
			\node [] (e) at (8, -2.5) {\color{blue} $[xy \vert xy]$ \color{black} } ;
			\draw [line width = 0.5mm, blue] (8,6) -- (8,9);
			\draw [line width = 0.5mm, blue] (8,6) -- (11,9);
			
			\draw [line width = 0.5mm, yellow] (8.9,-2) -- (8.9,9);
			\draw [line width = 0.5mm, yellow] (9.1,-2) -- (9.1,9);
			\node [] (g) at (9, -3.5) {\color{yellow} \large $[x \vert y \vert xy]$ \color{black}}; 
			\node [] (g) at (9, -4.5) {\color{yellow} \large $[y \vert x \vert xy]$ \color{black}}; 
			
			\draw [line width = 0.5 mm, green, dashed] (8,4) -- (9,5);
			\node[circle,draw, green, line width=0.5 mm, fill=white, outer sep=0pt, inner sep=2pt] (c) at (9,5){}; 
			\draw [-stealth, green, line width = 0.5 mm] (8.9,5.9) -- (8,5.9);
			\draw [-stealth, green, line width = 0.5 mm] (9.1,6.1) -- (8,6.1);
			\draw [-stealth, green, line width = 0.5 mm] (8.9,6.9) -- (8,6.9);
			\draw [-stealth, green, line width = 0.5 mm] (9.1,7.1) -- (8,7.1);
			\draw [-stealth, green, line width = 0.5 mm] (8.9,7.9) -- (8,7.9);
			\draw [-stealth, green, line width = 0.5 mm] (9.1,8.1) -- (8,8.1);
			\node [] (a) at (8.75, 8.5) {\color{green} \large $\mathbf{d_1} $ \color{black}}; 
			
			\draw [line width = 0.5mm, blue] (10,-2) -- (10,5);
			\draw [line width = 0.5mm, blue] (3,-2) -- (10,5);
			\node [] (e) at (10, -5.5) {\color{blue} $[x \vert y \vert y \vert x]$ \color{black} } ;
			\draw [line width = 0.5mm, blue] (10,7) -- (10,9);
			\draw [line width = 0.5mm, blue] (10,7) -- (12,9); 
			
			\draw [-stealth, green, line width = 0.5 mm] (10, 5) -- (9, 5);
			\node [] (a) at (9.75, 5.5) {\color{green} \large $\mathbf{d_1} $ \color{black}}; 
			
			\draw [-stealth, rose, line width = 0.5 mm] (9,4) -- (8,4);
			\draw [-stealth, rose, line width = 0.5 mm] (8,3) -- (7,3);
			\draw [-stealth, rose, line width = 0.5 mm] (7,2) -- (6,2);
			\node [] (a) at (7, 2.5) {\color{rose} \large $\mathbf{d_2} $ \color{black}}; 
			\draw [-stealth, rose, line width = 0.5 mm] (6,1) -- (5,1);
			\draw [-stealth, rose, line width = 0.5 mm] (5,0) -- (4,0);
			\draw [-stealth, rose, line width = 0.5 mm] (4, -1) -- (3, -1);
			
		\end{tikzpicture}
		\caption{Differentials in the twisted bar spectral sequence computing $H_\star K_{2 \sigma}$}
		\label{fig:2sigmaDiffs}
	\end{center} 
\end{figure}
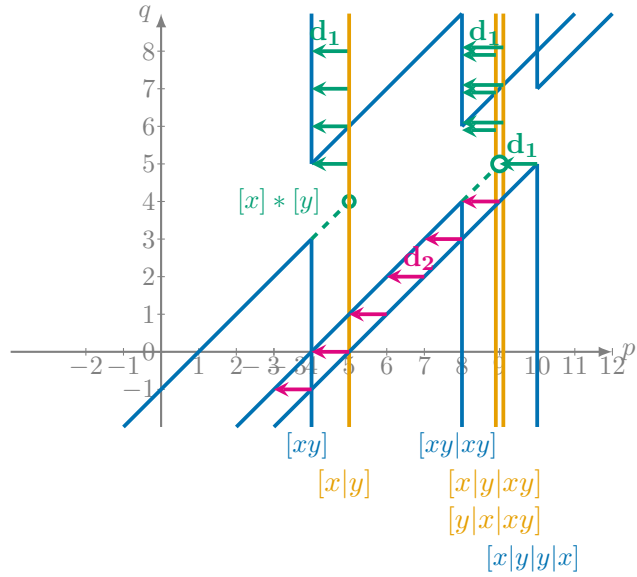 To compute the $d_1$-differential in this spectral sequence, consider the cofiber sequence
$$S^0 \xrightarrow[]{\quad a \quad} S^\sigma \xrightarrow{\quad \quad} Ca \simeq \Sigma {C_2}_+.$$
This induces a long exact sequence in homology involving
 $$H_\star S^0 \xrightarrow[]{\quad \cdot a \quad} H_\star S^\sigma \xrightarrow{\quad \quad} H_\star ({C_2}_+) $$ as shown in Figure \ref{Fig:d1wExt}.
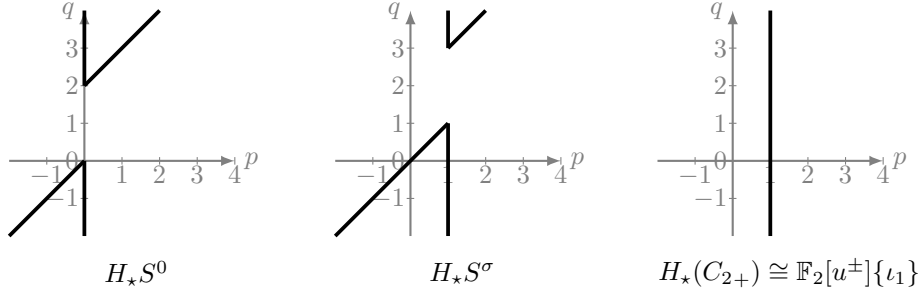
\begin{figure}[ht]
	\begin{center}
		\begin{multicols}{3}
			\begin{tikzpicture}[scale=0.5]
				
				\coordinate (Origin)   at (0,0);
				\coordinate (XAxisMin) at (-2,0);
				\coordinate (XAxisMax) at (4,0);
				\coordinate (YAxisMin) at (0,-2);
				\coordinate (YAxisMax) at (0,4);
				\draw [thick, gray,-latex] (XAxisMin) -- (XAxisMax) node[anchor=west] {$p$};
				\draw [thick, gray,-latex] (YAxisMin) -- (YAxisMax) node[anchor=east] {$q$};
				
				\foreach \y in {-1,0,1,2,3}\draw [gray] (-0.1, \y) -- (0.1, \y) node[anchor=east] {$\y$};
				
				\foreach \x in {-1,1,2,3,4}\draw [gray] (\x, -0.1) -- (\x, 0.1) node[anchor=north] {$\x$};
				
				\draw [line width = 0.5mm] (0,-2) -- (0,0);
				\draw [line width = 0.5mm] (-2,-2) -- (0,0);
				\node [] (a) at (1.5, -3) {$H_\star S^0$ \color{black} } ;
				\draw [line width = 0.5mm] (0,2) -- (0,4);
				\draw [line width = 0.5mm] (0,2) -- (2,4);
				
			\end{tikzpicture}
		
			\begin{tikzpicture}[scale=0.5]
				
				\coordinate (Origin)   at (0,0);
				\coordinate (XAxisMin) at (-2,0);
				\coordinate (XAxisMax) at (4,0);
				\coordinate (YAxisMin) at (0,-2);
				\coordinate (YAxisMax) at (0,4);
				\draw [thick, gray,-latex] (XAxisMin) -- (XAxisMax) node[anchor=west] {$p$};
				\draw [thick, gray,-latex] (YAxisMin) -- (YAxisMax) node[anchor=east] {$q$};
				
				\foreach \y in {-1,0,1,2,3}\draw [gray] (-0.1, \y) -- (0.1, \y) node[anchor=east] {$\y$};
				
				\foreach \x in {-1,1,2,3,4}\draw [gray] (\x, -0.1) -- (\x, 0.1) node[anchor=north] {$\x$};
				
				\draw [line width = 0.5mm] (1,-2) -- (1,1);
				\draw [line width = 0.5mm] (-2,-2) -- (1,1);
				\node [] (a) at (1.5, -3) {$H_\star S^\sigma$ \color{black} } ;
				\draw [line width = 0.5mm] (1,3) -- (1,4);
				\draw [line width = 0.5mm] (1,3) -- (2,4);
				
			\end{tikzpicture}
		
			\begin{tikzpicture}[scale=0.5]
				
				\coordinate (Origin)   at (0,0);
				\coordinate (XAxisMin) at (-2,0);
				\coordinate (XAxisMax) at (4,0);
				\coordinate (YAxisMin) at (0,-2);
				\coordinate (YAxisMax) at (0,4);
				\draw [thick, gray,-latex] (XAxisMin) -- (XAxisMax) node[anchor=west] {$p$};
				\draw [thick, gray,-latex] (YAxisMin) -- (YAxisMax) node[anchor=east] {$q$};
				
				\foreach \y in {-1,0,1,2,3}\draw [gray] (-0.1, \y) -- (0.1, \y) node[anchor=east] {$\y$};
				
				\foreach \x in {-1,1,2,3,4}\draw [gray] (\x, -0.1) -- (\x, 0.1) node[anchor=north] {$\x$};
				
				\draw [line width = 0.5mm] (1,-2) -- (1,4);
				\node [] (a) at (1.5, -3) {$H_\star ({C_2}_+) \cong \mF_2 [ u ^\pm] \{\iota_1\}$} ;
			\end{tikzpicture}
		\end{multicols}
		\caption{Computing a $d_1$-differential in the twisted bar spectral sequence} 
		\label{Fig:d1wExt}
	\end{center}
\end{figure}
The map
$$H_\star({C_2}_+) \to H_\star (S^{\sigma - 1})$$ 
is the map depicted in Figure \ref{Fig:d1diff}.
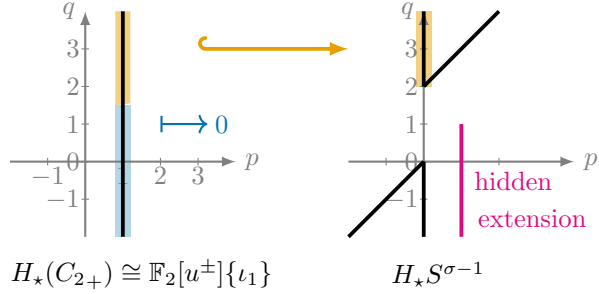
\begin{figure}[ht]
	\begin{center}
		\begin{tikzpicture}[scale=0.5]
			\draw[fill=yellow!50, draw = none]  (0.8,1.5) -- (0.8,4) -- (1.2,4) -- (1.2,1.55);
			\draw[fill=yellow!50, draw = none]  (8.8,2) -- (8.8,4) -- (9.2,4) -- (9.2,2);
			\draw [line width = 0.5mm, right hook-latex, yellow] (3, 3) -- (7, 3); 
			
			\draw[fill=blue!30, draw = none]  (0.8,1.5) -- (0.8,-2) -- (1.2,-2) -- (1.2,1.5);
			\draw[|->, thick, blue] (2,1) -- (3.25,1);
			\node (b) at (3.75,1) {\color{blue} $0$ \color{black}};

			\coordinate (Origin)   at (0,0);
			\coordinate (XAxisMin) at (-2,0);
			\coordinate (XAxisMax) at (4,0);
			\coordinate (YAxisMin) at (0,-2);
			\coordinate (YAxisMax) at (0,4);
			\draw [thick, gray,-latex] (XAxisMin) -- (XAxisMax) node[anchor=west] {$p$};
			\draw [thick, gray,-latex] (YAxisMin) -- (YAxisMax) node[anchor=east] {$q$};
			
			\foreach \y in {-1,0,1,2,3}\draw [gray] (-0.1, \y) -- (0.1, \y) node[anchor=east] {$\y$};
			
			\foreach \x in {-1,1,2,3}\draw [gray] (\x, -0.1) -- (\x, 0.1) node[anchor=north] {$\x$};
			
			\draw [line width = 0.5mm] (1,-2) -- (1,4);
			\node [] (a) at (1.5, -3) {$H_\star ({C_2}_+) \cong \mF_2 [ u ^\pm] \{\iota_1\}$} ;
			
			\coordinate (Origin)   at (9,0);
			\coordinate (XAxisMin) at (7,0);
			\coordinate (XAxisMax) at (13,0);
			\coordinate (YAxisMin) at (9,-2);
			\coordinate (YAxisMax) at (9,4);
			\draw [thick, gray,-latex] (XAxisMin) -- (XAxisMax) node[anchor=west] {$p$};
			\draw [thick, gray,-latex] (YAxisMin) -- (YAxisMax) node[anchor=east] {$q$};
			
			\foreach \y in {-1,0,1,2,3}\draw [gray] (8.9, \y) -- (9.1, \y) node[anchor=east] {$\y$};
			
			\foreach \x in {8,9,10,11}\draw [gray] (\x, -0.1) -- (\x, 0.1) node[anchor=north] {};
			
			\draw [line width = 0.5mm] (9,-2) -- (9,0);
			\draw [line width = 0.5mm] (7,-2) -- (9,0);
			\node [] (a) at (9.5, -3) {$H_\star S^{\sigma - 1}$ \color{black} } ;
			\draw [line width = 0.5mm] (9,2) -- (9,4);
			\draw [line width = 0.5mm] (9,2) -- (11,4);
			
			\draw [line width = 0.5mm, rose] (10,1) -- (10,-2);
			\node [] (a) at (11.5, -0.5) {\color{rose} hidden \color{black} } ;
			\node [] (a) at (12, -1.5) {\color{rose} extension \color{black} } ;
			
		\end{tikzpicture}
		\caption{$RO(C_2)$-graded twisted bar spectral sequence $d_1$-differential with hidden extension}
		\label{Fig:d1diff}
	\end{center}
\end{figure}

We have shown that the $d_1$-differentials marked in green in Figure \ref{fig:2sigmaDiffs} both exist and have the behavior of the map in Figure \ref{Fig:d1wExt}. We also know from Theorem \ref{Thm:signed} that 
$$ H_\star K_{2 \sigma} \cong E[e_{2 \sigma}, \bar{\alpha}_{(j_1)} \circ \bar{\alpha}_{(j_2)} ]$$
where $j_1 \leq j_2$. 

Since the $RO(C_2)$-graded homology of $K_{2\sigma}$ is free over the $RO(C_2)$-graded homology of a point, all copies of $H_\star {C_2}_+$ appearing on the $E^1$-page must either be killed off or used in shifting the representation degree of the $RO(C_2)$-graded homology of a point, similar to the equivariant degree shifting differential $d_1$ and hidden extension of Figure \ref{Fig:d1diff}. 

We also know the underlying integer graded homology of $K_2,$ and have both the forgetful map $H_\star K_{2 \sigma} \to H_2 K_2$ and the fixed point map $H_\star K_{2 \sigma} \to H_* (K_{2\sigma})^{C_2} \cong H_* (K_2 \times K_1 \times K_0).$ Given that $H_\star K_{2 \sigma}$ is free and in the underlying nonequivariant case $[xy \, \vert \, xy]$ is killed by a $d_1$ differential (all generators of $H_\star K_{2 \sigma}$ have nontrivial underlying homology), the entire double cone supported by the twisted bar representative $[xy \, \vert \, xy ]$ must be hit by a differential. 

There is a $d_1$-differential and hidden extension shifting the double cone supported by $[ xy \, | \, xy]$ up by representation degree $\sigma$ so that by the $E^2$-page the double cone is in fact in representation degree $\rho ( \vert x \vert +  \vert y \vert) + \rho + \sigma = \rho (\rho + \sigma) + \rho + \sigma = 4 \rho + \sigma$. We hypothesize there is a $d_2$-differential induced by a $d_1$-differential supported by $[x \, \vert \, y \,  \vert \, y  \, \vert \, x].$ We notice that $[x \, \vert \, y \, \vert \, y \, \vert \, x]$ is a  norm of $[xy \, \vert \, xy].$ We expect such norms play an important role in governing the structure of all the higher nontrivial differentials.

As one goes farther along in the spectral sequence, considering cycles supported by twisted bar representatives such as $[xy \, \vert \, xy \, \vert \, xy]$ and $[xyz \, \vert \, xyz],$ which must all be killed off in order to recover the correct underlying homology, we see that arbitrarily long equivariant degree shifting differentials are required in order to arrive at the answer given by Theorem \ref{Thm:signed}. We conjecture all such cycles are killed by differentials induced by a norm structure on the twisted bar spectral sequence. 

\section{Related questions}\label{sec:futurdir}

We describe a few questions of immediate interest given the results of this paper. 

\subsection{Twisted Tor and the $RO(C_2)$-graded twisted bar spectral sequence}
In the $C_2$-equivariant setting, the $RO(C_2)$-graded homology of each signed delooping, $K_{V + \sigma},$ of an equivariant Eilenberg-MacLane space, $K_{V},$ also independently arises as the result of a $C_2$-equivariant twisted Tor computation. This can be seen by taking the model of $\sigma$-delooping defined in \cite{Hill2019}. In this model, $A$ is an $E_\sigma$-algebra and
$$B^\sigma(A) = B(A, \text{Map}(C_2, A), \text{Map}(C_2, *)), $$
where the action of $\text{Map}(C_2,A)$ on $A$ is via the $E_\sigma$-structure (\cite{Hill2019}, Definition 5.10). In the case that $A$ has $R$-free homology, Hill constructs yet another twisted bar spectral sequence with $E^2$-page
$$E_2^{s, \underline{\star}} = \text{Tor}_{-s}^{N_e^{C_2}(i_e^* R_*(i_e^* A))} (R_{\underline{\star}} (\text{Map}(C_2, X)), R_{\underline{\star}} (A) ) \Rightarrow R_{\underline{\star} - s} (B^\sigma (A)) $$
(\cite{Hill2019}, Theorem 5.11). Computations with this spectral sequence are complicated and the literature lacks substantial examples. However, it does have a twisted Tor functor as its $E^2$-page and thus it would be interesting to compare with our computations. 

One notable feature of the nonequivariant computation of $H_* K(\mF_p, *)$ is that the integer graded bar spectral sequences collapse on the $E^2$-page \cite{Wilson1982}. In contrast, we saw that the $RO(C_2)$-graded twisted bar spectral sequences computing $H_\star K_{* \sigma}$ have arbitrarily long differentials in Section \ref{Subsec:highertwdiffs}. Thus under favorable circumstances, we hope to formulate a twisted bar spectral sequence with $E^2$-page a twisted Tor functor arising as a derived functor of the twisted product of $H \umF_2$-modules, which collapses in the relevant cases of $H_\star K_{*\sigma}.$ 

Given our computation of $H_\star K_{* \sigma},$ such a twisted Tor over an exterior algebra should have the property that
$$\text{Tor}_{tw}^{E[x]} \cong E[\sigma x] \otimes \Gamma [\cN_e^{C_2} x].$$ 

\subsection{Global Hopf rings}

In their work computing the integer graded homology of classical nonequivariant Eilenberg-MacLane spaces, Ravenel and Wilson obtain a global statement. Specifically,
\begin{thm}[Ravenel-Wilson \cite{Wilson1982}]
	$H_* K_*$ is the free Hopf ring on $H_* K_0 = H_* [\mF_p],$ $H_* K_1,$ and $H_* \mC P^\infty \subset H_* K_2$ subject to the relation that $e_1 \circ e_1 = \beta_1.$ 
\end{thm}
\noindent It is natural to ask if a similar statement be obtained in the $C_2$-equivariant case, and in that case, what specifically, is the global structure of the Hopf rings that do arise. One may also ask how the Hopf rings here relate to Hill and Hopkins' work extending Ravenel and Wilson's construction of a universal Hopf ring over $MU^*$ to $C_2$-equivariant homotopy theory \cite{HillHopkins2018}.

\subsection{Stabilizing to the $C_2$-dual Steenrod algebra}
Besides understanding a global version of the unstable story, it also remains to fully understand how the the unstable answer for $H_\star K_V$ stabilizes to give the $C_2$-equivariant dual Steenrod algebra,
$$ \mathcal{A}_\star^{C_2} = H \umF_2 [\tau_0, \tau_1, \cdots, \xi_1, \xi_2, \cdots] / (\tau_i^2 = (u + a \tau_0)\xi_{i + 1} + a \tau_{i + 1}).$$
By Hu-Kriz's construction of the $C_2$-equivariant dual Steenrod algebra \cite{HuKriz2001}, we should homology suspend $\bar{\beta}_{(i)}$ to define
$$\xi_i \in H_{(2^i - 1) \rho} H$$
and $\bar{\alpha}_{(i)}$ to define
$$ \tau_i \in H_{2^i \rho - \sigma} H. $$
However, it is not at all clear what an arbitrary element in $H_\star K_V$  should stabilize to in the the $C_2$-equivariant dual Steenrod algebra. Additionally, there is the interesting problem of understanding how the stable relation $\tau_i^2 = (u + a \tau_0)\xi_{i + 1} + a \tau_{i + 1} $ arises unstably. We look forward to studying these questions in forthcoming work.  

\bibliographystyle{amsalpha}

\bibliography{RWhopfHF2em}

\end{document}